%% file: main.tex
\documentclass[11pt,reqno]{amsart}

\usepackage{lmodern}
\usepackage[T1]{fontenc}
\usepackage[utf8]{inputenc}
\usepackage[english]{babel}
\usepackage{microtype}

\usepackage{amsmath, amssymb, amsfonts, amsthm, mathtools}
\allowdisplaybreaks              % long align environments may split across pages
\numberwithin{equation}{section}

\usepackage{graphicx}
\graphicspath{{./images/}}
\usepackage{xcolor}
\usepackage{comment}

\theoremstyle{plain}
\newtheorem{theorem}{Theorem}[section]
\newtheorem{corollary}[theorem]{Corollary}
\newtheorem{lemma}[theorem]{Lemma}
\newtheorem{prop}[theorem]{Proposition}

\theoremstyle{definition}
\newtheorem{definition}[theorem]{Definition}

\theoremstyle{remark}
\newtheorem{rem}[theorem]{Remark}

\usepackage[
  colorlinks=true,
  linkcolor=blue!60!black,
  citecolor=green!50!black,
  urlcolor=blue!60!black,
  breaklinks=true
]{hyperref}
\usepackage[capitalise,nameinlink]{cleveref}   

\title[Mixing for Blaschke Products]{Mixing for Free Semigroup Actions of Blaschke Products on the Circle}

\author{Ekaterina Shchetka}
\address{Department of Mathematics, Rice University, Houston, TX 77004, USA}
\email{katia.shchetka@rice.edu}


\thanks{The author was partially supported by NSF grant DMS 2404309}
\makeatletter
\renewcommand{\MR}[1]{}
\makeatother

\begin{document}

\begin{abstract}
We study mixing at a double exponential rate on analytic observables and ask
how much of a map is remembered by its rate of mixing. For finite Blaschke products of the circle with a fixed point in the unit disk, and for the free
semigroup actions they generate, we give a complete classification: the rate of
mixing (no mixing, exponential, or double exponential) is determined by the
multiplier of the generators at that fixed point, the invariant measure being
the harmonic measure with a pole there. 

In the double exponential regime, the
exponent equals $\log p$, where $p$ is the minimal local degree of the generators at the fixed point, and
we show that this value is sharp. Consequently, the rate is not rigid: it is not
stable under $C^1$-perturbations and does not imply $C^1$-conjugacy to affine
models. Rigidity holds when the exponent is maximal for the degree: a map of degree
$q$ whose exponent attains $\log q$ is M\"obius conjugate to an affine model.
\end{abstract}

\maketitle

\tableofcontents

%==============================================================================

\input{intro_arxiv}
\input{hyperfunctions_final}

\input{quasi-nilpotent}
\input{harmonic_fixed}

%==============================================================================
%  Appendices
%==============================================================================
\appendix
\input{Invariance_of_Lebesgue}
\input{some_theta_any_theta}
\input{Mixing_Hierarchy}
\input{no_uniform_mixing_L_2}

%==============================================================================
%  Bibliography
%==============================================================================
\bibliographystyle{amsalpha}
\bibliography{Katia}

% biblatex users: replace the two lines above with
%\printbibliography
\end{document}

%% file: intro_arxiv.tex
\section{Introduction}\label{sec:intro}
 
Mixing states that correlations between observables decay under the action of a semigroup (or a group) of measure-preserving transformations.
The rate of mixing is not a property of the system alone: it depends on the class of
observables against which it is measured, and it improves as that class is
narrowed. On square-integrable observables there is no uniform rate (Appendix~\ref{no-uniform-mixing}). On H\"older observables, the classical result states that for uniformly expanding maps of
the circle, correlations decay exponentially. This goes back to Ruelle, Sinai, and Bowen \cite{Rue68, Sin72, Bow75}. The mechanism underlying this comes from a spectral gap for the
transfer operator on a suitable Banach space, and the rate is defined by the modulus of its
second eigenvalue \cite{Bal00}.
 
On analytic observables the transfer operator becomes compact, and its spectrum
is a discrete set of eigenvalues, the \textit{Pollicott--Ruelle resonances},
which again determine the rate. For analytic expanding maps of the circle these spectra are described in \cite{BJS17} and \cite{BN19}; they can also be estimated by rigorous numerics \cite{BNTC25}. This suggests looking at the extreme case, when
the resonance spectrum is trivial: restricted to the observables of zero mean the operator is \textit{quasi-nilpotent}, its
spectral radius equals zero, and no exponential bound is optimal. Such systems
exist, and the simplest one is the Bernoulli doubling map $\phi(z)=z^2$. On the space of analytic observables, the correlations decay \textit{double exponentially} in $n$ (see
Section~\ref{motivating example}). In the mixing hierarchy, mixing at a double exponential rate on analytic observables implies exponential mixing on H\"older observables (Appendix~\ref{hierarchy}). Double exponential rates of decay are closely related to \textit{dissipation time} under stochastic perturbations in \cite{FNW04},
\cite{FI19}, and \cite{T25}.

In the regime of double exponential decay the resonances carry no information: apart from the trivial
eigenvalue $1$, carried by the constant observables, they are all
equal to zero. What survives is the \textit{rate}
itself. This is the question we address.
 
\vspace{3mm}
\textbf{Question.} How much of a map is remembered by its rate of mixing on analytic observables?
\vspace{3mm}
 
We answer it completely for finite Blaschke products of the circle with a fixed
point in the unit disk, and for the free semigroup actions they generate. \textit{Qualitatively}, the rate remembers exactly the
multiplier of the generators at that fixed point: whether it has modulus one,
modulus strictly between zero and one, or is zero (Theorem~\ref{classification}).
A global property of the dynamics on the circle is thus determined by the local
behavior at an interior fixed point. \textit{Quantitatively}, in the double
exponential regime the exponent equals $\gamma=\log p$, where $p$ is the local
degree at that fixed point, and this value is sharp
(Theorem~\ref{thm:sharpness}). Consequently, the rate cannot remember more: for each degree $q$
there are many maps with $p=2$, and they are not even $C^1$-conjugate to the
affine (in the angular
coordinate) model of their degree, so the rate is not rigid
(Section~\ref{examples}). The exponent does determine the map once the degree is
fixed as well: a map of degree $q$ whose exponent attains the maximal possible
value $\gamma=\log q$ is M\"obius conjugate to the affine model
$z\mapsto e^{i\psi}z^{q}$
(Corollary~\ref{cor:rigidity fixed degree}).
 
We work with free semigroup actions rather than with a single map because the
natural object here is a family of maps applied in an arbitrary order, as in
random dynamics and iterated function systems; the word length $|a|$ then plays
the role of time $n$. Several generators also bring in features that a single map cannot show: the invariant
replacing the spectral radius is the \textit{joint} spectral radius of
$\mathcal{L}_1,\dots,\mathcal{L}_N$, which we compute (Theorem~\ref{prop:jsr}), and in the
double exponential regime the generators span a non-unital algebra of compact
quasi-nilpotent operators, which need not be commutative
(Section~\ref{spec-qnilp-res}).
 
\subsection{Setting}\label{sec:setting}

Let $\mathbb{S}^1=\{z\in\mathbb{C} : |z|=1\}$ denote the unit circle, and let $\mu$ denote the normalized Lebesgue measure on $\mathbb{S}^1$, given in the
complex coordinate by
\begin{equation*}
d\mu(z):=\frac{1}{2\pi i}\frac{dz}{z} .
\end{equation*}
Let $\phi_1, \dots, \phi_N$
be non-constant analytic functions on the open unit disk that admit continuous
extensions to the closed unit disk and map the unit circle to itself, that is,
\[
|\phi_j(z)|=1 \quad \text{for} \quad |z|=1, \quad1\le j\le N.
\]
By a theorem of Fatou \cite{Fat23}, these are exactly the finite \textit{Blaschke products}
 \[
\mathcal{B}:=\left\{ e^{i\psi}\prod_{i=1}^q\frac{z-\lambda_i}{1-\overline{\lambda_i}z}\,:\, \psi\in[0,2\pi),|\lambda_i|<1,  i=1,\dots,q, \,\text{ for }\, q\in\mathbb{Z}_{\ge1}\right\}.
 \]
These have degree $q\ge1$. For $q=1$ one gets the
automorphisms of $\mathbb{D}$. Those of degree at least two with a fixed point in $\mathbb{D}$ are
uniformly expanding maps of the circle (see \cite{Mar83}), hence exponentially
mixing on H\"older observables \cite{Rue68}. For a fixed $q\ge 2$, McMullen studied the
topological dynamics of these maps in \cite{McM10}. Further statistical properties of the boundary dynamics, including
central limit theorems for iterates of a single inner function are studied by Ivrii and
Urba\'nski in \cite{IU23}.

Throughout we assume that $\phi_1,\dots,\phi_N$ have a \textit{common fixed
point} $w$ in the open unit disk $\mathbb{D}$, and we let
\begin{equation}\label{eq:harmonic measure}
d\nu_w(z):=\frac{1-|w|^{2}}{|z-w|^{2}}\,d\mu(z)
\end{equation}
be the harmonic measure on $\mathbb{S}^1$ with pole at $w$. By
Corollary~\ref{cor:harmonic measure preserving},
$\phi_j$ preserves $\nu_w$ if and only if $\phi_j(w)=w$. For $w=0$ we have
$\nu_0=\mu$, and the condition is $\phi_j(0)=0$ (Lemma~\ref{lm:Leb meas preserv}). We write
 \[
 \mathcal{B}^0:=\{\phi\in\mathcal{B}: \phi(0)=0\}
 \]
for the maps normalized in this way, so that $\mathcal{B}^0$ is exactly the class
of Lebesgue measure preserving finite Blaschke products.

The collection of functions $\phi_j$, $j=1,\dots,N$, generates an action of the
free semigroup $\mathbb{F}_N^+$ on the circle defined by
\begin{align*}
\tau(a) z=\phi_{i_n}\circ \phi_{i_{n-1}}\circ \dots \circ \phi_{i_1} (z), \quad a=i_1i_2\dots i_n \in \mathbb{F}_N^+, \quad z\in\mathbb{S}^1,
\end{align*}
so that $\tau(ab)=\tau(b)\circ\tau(a)$.

In \cite{Pommerenke}, Pommerenke proved that if none of the $\phi_j$ are
rotations, then $\tau$ is \textit{(strongly) mixing}; that is, for any
$ f,g\in L_2(\mathbb{S}^1, \nu_w)$, the \textit{correlation function}
\[
\mathcal{C}_{f,g}(a):=\left|\int_{\mathbb{S}^1}(f\circ\tau(a))\,\overline{g}\,d\nu_w -\int_{\mathbb{S}^1}f\,d\nu_w \int_{\mathbb{S}^1}\overline{g}\,d\nu_w  \right| \to 0 \quad \text{as} \quad |a|\to\infty,
\]
where $|a|$ denotes the word length of $a\in \mathbb{F}_N^+$.

We study the \textit{quantitative mixing} for the action $\tau$ on the spaces of
analytic functions $\mathcal{H}_\theta$.

\begin{definition}[Analytic functions]\label{analytic space} Let $0<\theta<1$.
 The space of analytic functions $\mathcal{H}_\theta$  is defined as the closure of all linear combinations of $z^j, j\in \mathbb{Z}$, with respect to the norm
 \[
 \|f\|_\theta:=\sqrt{\sum_{k\in\mathbb{Z}} |f_k|^2\theta^{-2|k|}},
 \]
 where $f_k$ are the Fourier coefficients of $f$.
 We denote by $\mathcal{H}_\theta^0$ the subspace of $\mathcal{H}_\theta$ orthogonal to constants, that is, those with $f_0=0$.
\end{definition}

Since the density of $\nu_w$ is bounded above and below on $\mathbb{S}^1$, we
have $L_2(\mathbb{S}^1,\nu_w)=L_2(\mathbb{S}^1,\mu)$ with equivalent norms, and
$\mathcal{H}_\theta$ is contained in both. 

\begin{definition}[Mixing at a double exponential rate]\label{definition of double exponential mixing}
  An action  $\tau$ is said to be mixing at a double exponential rate if for some $0<\theta<1$ there exist $\beta>0$, $\gamma>0$, and $C>0$ such that for any $f, g\in \mathcal{H}_\theta$
and any $a\in \mathbb{F}_N^+$
\[
\mathcal{C}_{f,g}(a) \le C{e^{-\beta{e^{\gamma|a|}}}}\|f\|_\theta\|g\|_\theta.
\]
\end{definition}
By Lemma~\ref{lm:theta independence}, Definition~\ref{definition of double exponential mixing} does not depend on $\theta$: if it holds for one, it holds for every $0<\theta<1$, with the same double exponential rate $\gamma$.

Let $C^\alpha$ denote the space of H\"older continuous functions with $\alpha\in(0,1]$.
\begin{definition}[Exponential mixing]\label{definition of exponential mixing}
An action $\tau$ is said to be exponentially mixing on $C^\alpha$ if there exist $\kappa>0$ and $C>0$ such that for any $f,g\in C^\alpha$ and any $a\in \mathbb{F}_N^+$
\[
\mathcal{C}_{f,g}(a) \le Ce^{-\kappa|a|}\|f\|_{C^{\alpha}}\|g\|_{C^{\alpha}}.
\]
\end{definition}

\subsection{Results}\label{ch1:results}

Below, we provide a complete classification of the mixing rate for the action
$\tau$, distinguishing between \textit{exponential} and \textit{double
exponential} behavior. The invariant is the multiplier $\phi_j'(w)$ at the
common fixed point.

\begin{theorem}\label{classification} Let $\tau$ be as above. Then
 \begin{enumerate}
     \item $\tau$ is not mixing if and only if $|\phi_j'(w)|=1$ for some $ 1\le j\le N$,
    \item $\tau$ is exponentially mixing if and only if  $|\phi_j'(w)|<1$ for all $ 1\le j\le N$,
    \item $\tau$ is mixing at a double exponential rate if and only if $\phi_j'(w)=0$ for all $ 1\le j\le N$.
\end{enumerate}
\end{theorem}

Note that, by the Schwarz lemma, $|\phi_j'(w)|$ cannot be greater than $1$, and
equals $1$ precisely when $\phi_j$ is an automorphism of $\mathbb{D}$, that is,
an elliptic rotation about $w$. Thus, the three cases exhaust all possibilities.

In the normalized case when $w=0$ the constants $\gamma$, $\beta$, and $C$ in
Definitions~\ref{definition of double exponential mixing} and
\ref{definition of exponential mixing} are computed explicitly (see equations
\eqref{constants for double exponential mixing} and
\eqref{constants for exponential mixing}); in particular $\beta=-\log \theta$
and $\gamma=\log p$, where $p\ge2$ is the minimal local degree of the generators
at the fixed point. The exponent $\gamma$ is sharp, that is it cannot be replaced by any larger one (see Theorem~\ref{thm:sharpness}).

In the case $N=1$, we obtain the classification for a single Blaschke product.

\begin{corollary}\label{intro:classification on S1}
Let $\phi$ be a finite Blaschke product with a fixed point $w\in\mathbb{D}$, and
let $\nu_w$ be the harmonic measure with pole at $w$. Then
\begin{itemize}
    \item $\phi$ is not mixing if and only if $|\phi'(w)|=1$,
    \item $\phi$ is exponentially mixing if and only if $|\phi'(w)|<1$,
    \item $\phi$ is mixing at a double exponential rate if and only if $\phi'(w)=0$.
\end{itemize}
\end{corollary}
Equivalently: $\phi$ fails to be mixing exactly when it is an elliptic
automorphism; it mixes exponentially exactly when its fixed point in the disk is
attracting but not critical; and it mixes at a double exponential rate exactly
when that fixed point is critical.

In complex dynamics, $w$ is called a \textit{neutral} (or
\textit{indifferent}) fixed point if $|\phi'(w)|=1$, \textit{attracting} if
$|\phi'(w)|<1$, and \textit{super-attracting} if $\phi'(w)=0$. Therefore, we see
that mixing at a double exponential rate on the circle, a global dynamical
property, is determined by the local behavior of $\phi$ at its fixed point
inside the unit disk.

\subsection{Ideas of Proof}\label{sec:ideas}
All the proofs are carried out in the normalized case $w=0$, where $\nu_w=\mu$
and $\phi_j\in\mathcal{B}^0$. Conjugating by the M\"obius transformation
$h(z)=\frac{z+w}{1+\overline{w}z}$ reduces the general case to this one. Since it
preserves the multipliers and the local degrees at the fixed point, the rate
$\gamma$ is preserved. The details are in
Section~\ref{Other invariant measures}.
 
The first part of Theorem~\ref{classification} follows from the Schwarz lemma,
since $|\phi_j'(0)|=1$ is equivalent to $\phi_j$ being a rotation, and rotations
are never mixing (see Section~\ref{rotations}). The second part follows from the
classical theory of uniformly expanding maps \cite{Rue68}, since
$|\phi_j'(0)|<1$ implies that $\phi_j$ is an analytic, uniformly expanding map
of the circle \cite{Mar83}. Note that because the generator family is finite, the expansion constants may be chosen uniformly in $j$. In Section~\ref{Exponential mixing via complex analysis} we give
an alternative proof for exponential mixing on analytic spaces based on an application of the Schwarz lemma.
 
In the proof of the third part, we use a \textit{functional analysis approach}, studying
properties of the operator $\mathcal{L}$ on the space of \textit{hyperfunctions}, which is the dual space of
analytic functions. In fact the upper bounds we prove are stronger than the
classification requires: they give the same rate of decay already when only one
of the two observables is analytic and the other is allowed to be a
hyperfunction. Restricting both observables to the analytic space gives the
statements above. Rather than estimating $\|\mathcal{L}\|$ and iterating the
estimate (an approach typical for obtaining exponential decay), we estimate
$\|\mathcal{L}^n\|$ directly. We prove that $\|\mathcal{L}^n\|$ decays double
exponentially in $n$ (see Theorem~\ref{norm bound}) and provide explicit
estimates for $C,\beta$, and $\gamma$. As a consequence, we show that when
restricted to the subspace orthogonal to constants, $\mathcal{L}$ is a compact,
\textit{quasi-nilpotent operator} (that is, its spectral radius is equal to
zero, see Theorem~\ref{prop:jsr}). In other words, the resonance
spectrum is trivial. For $N>1$ we prove more: every noncommutative polynomial %with zero constant term
in
$\mathcal{L}_1,\dots,\mathcal{L}_N$ is quasi-nilpotent, so the generators span a
quasi-nilpotent non-unital algebra (Theorem~\ref{prop:jsr}).
This does not follow from the quasi-nilpotency of the individual generators,
since a composition of quasi-nilpotent operators need not be quasi-nilpotent.
 
One feature of our approach is that, when we prove that a map does \textit{not}
mix at a double exponential rate, we provide an \textit{explicit lower bound} on
the exponential decay rate (see Lemma~\ref{lower bounds on exponential mixing}). The
same technique gives the sharpness of the exponent $\gamma=\log p$, and with it
the rigidity statement in fixed degree.
 
\subsection{Organization of the Paper}
Section~\ref{motivating example} discusses the model case of the Bernoulli
doubling map completely, including the spectral picture on the space of
hyperfunctions. Sections~\ref{proof} and \ref{one-more-proof} prove
Theorem~\ref{classification} in the normalized case via the norm bounds of
Theorem~\ref{norm bound}. Section~\ref{sharp-and-rigid} provides sharpness of the exponent and the rigidity statement in fixed degree.  Section~\ref{examples} shows that the maps mixing at a
double exponential rate are nowhere dense in the $C^1$ topology and are not
$C^1$-conjugate to affine models. Section~\ref{spec-qnilp-res} covers the joint spectral radius. Section~\ref{Other invariant measures}
carries out the reduction of harmonic measure case to the Lebesgue measure case.
Appendix~\ref{app:Lebesgue measure preserving} proves the invariance criterion
for the harmonic measure; Appendix~\ref{some theta any theta} shows that the
definition of mixing at a double exponential rate does not depend on $\theta$;
Appendix~\ref{hierarchy} proves the implication from double exponential to
exponential mixing; and Appendix~\ref{no-uniform-mixing} shows that no uniform 
quantitative rate is available on $L_2$.

%% file: hyperfunctions_final.tex
\section*{Acknowledgements}
The author would like to thank Alex Blumenthal, Jon Chaika, David Fisher,  Boris Hasselblatt, Asaf Katz, Osama Khalil, Roland Roeder, Julia Slipantschuk, Ralf Spatzier, Kurt Vinhage, Amie Wilkinson, and Maciej Zworski for their insightful questions, helpful discussions, and comments on this work.

\section{Motivating Example}\label{motivating example}

\begin{prop}\label{multiplication-by-2-is-DEM}
The Bernoulli doubling map $\phi : \mathbb{S}^1 \to \mathbb{S}^1$
\[
   \phi(z) = z^2
\]   is mixing at a double exponential rate on $\mathcal{H}_\theta$, i.e., for any $f,g\in\mathcal{H}_\theta$ and any $n\ge0$
\begin{equation}\label{DEM-for-Bernoulli}
\mathcal{C}_{f,g}(n)\le Ce^{-\beta e^{\gamma n}}\|f\|_\theta\|g\|_{\theta}
\end{equation}
with $C=\theta$ , $\beta=-\log \theta$, and $\gamma=\log 2$.
Moreover, the rate of mixing is sharp, i.e., there exist $f,g\in\mathcal{H}_\theta$ such that for some $C'>0$ and $\beta'>0$, and any $n\ge 0$
\[
\mathcal{C}_{f,g}(n)\ge C'e^{-\beta'e^{\gamma n}}\|f\|_\theta\|g\|_\theta.
\]
\end{prop}

\begin{proof}
In the Fourier basis $f(z) = \sum_{k\in\mathbb{Z}} f_k z^k$,
\[
   (f\circ\phi^n)(z) =f(z^{2^n})= \sum_{k\in\mathbb{Z}} f_k\, z^{2^nk}.
\]  
Now, let us compute the correlation function explicitly. Without loss of generality, we can assume that $f_0=0$, then for $f, g\in\mathcal{H}_\theta$, we compute the correlation function
\begin{align*}
\mathcal{C}_{f,g}(n)&=\left|\int_{\mathbb{S}^1}(f\circ\phi^n)\,\overline{g}\,d\mu \right| \\ &=\left| \sum_{0\neq m\in\mathbb{Z}}(f\circ \phi^n)_m\,\overline{g_m}\right| \\ &=\left| \sum_{0\neq k\in\mathbb{Z}}f_k\,\overline{g_{2^nk}}\right|.
\end{align*}
Next, we multiply and divide each term in the sum by $\theta^{(2^n+1)|k|}$ and use the Cauchy-Schwarz inequality
\begin{align*}
\mathcal{C}_{f,g}(n)&=\left| \sum_{0\neq k\in\mathbb{Z}}\theta^{(2^n+1)|k|}f_k\, \theta^{-|k|}\, \overline{g_{2^nk}}\, \theta^{-2^n|k|}\right| \\ &\le  \sup_{0\neq k\in\mathbb{Z}}\left(\theta^{(2^n+1)|k|}\right) \cdot\,\sqrt{\sum_{0\neq k\in\mathbb{Z}}|f_k|^2\theta^{-2|k|}}\cdot \sqrt{\sum_{0\neq k\in\mathbb{Z}}|g_{2^nk}|^2\theta^{-2\cdot2^n|k|}} \\ &\le\theta^{(2^n+1)}\|f\|_\theta\|g\|_\theta.
\end{align*}
Thus, we prove \eqref{DEM-for-Bernoulli} with \[C=\theta,\quad \beta=-\log \theta, \quad \gamma=\log 2.\]

Finally, we prove the sharpness of the rate. Let \[f(z)=z \quad\text{and}\quad g(z)=\sum_{0\neq k\in\mathbb{Z}}\frac{1}{|k|}\theta^{|k|}z^k.\] Note that 
\[
\|f\|_\theta=\theta^{-1} \quad \text{and}\quad \|g\|_{\theta}=\sqrt{2\sum_{k\ge1}{\frac{1}{k^2}}}=\frac{\pi}{\sqrt{3}}.
\]
For this choice of $f$ and $g$ the correlation function
\begin{align}\label{one character survives}
\mathcal{C}_{f,g}(n)&=\overline{g_{2^n}}\\&=\frac{\theta^{2^n}}{2^n}\nonumber\\&=e^{-\beta e^{\gamma n}-\gamma n} \nonumber\\&\ge e^{-\beta'e^{\gamma n}}\nonumber\\&=C'e^{-\beta'e^{\gamma n}}\|f\|_\theta\|g\|_\theta \nonumber,
\end{align}
for some $\beta'>\beta>0$ and $C'=\theta\frac{\sqrt{3}}{\pi}>0$.%<\theta=C$.
\end{proof}

The pullback under the Bernoulli doubling map $\phi(z)=z^2$ defines the \textit{Koopman} precomposition operator 
\begin{align*}
 \mathcal{L}:L_2(\mathbb{S}^1,\mu)&\to L_2(\mathbb{S}^1,\mu) \\
f&\mapsto f \circ  \phi.
\end{align*}
Next, we study the operator properties of $\mathcal{L}$ on different spaces. In particular, we will see that $\mathcal{L}$ has nice spectral properties on the space of \textit{hyperfunctions}.

\begin{definition}\label{def:hyperfunctions}
Let $0<\theta<1$. The space of hyperfunctions  $\mathcal{H}_{\theta^{-1}}$ is defined as the closure of all linear combinations of $z^j,j\in\mathbb{Z}$, with respect to the norm 
\[
 \|f\|_{\theta^{-1}}:=\sqrt{\sum_{k\in\mathbb{Z}} |f_k|^2\theta^{2|k|}},
 \]
 where $f_k$ are the Fourier coefficients of $f$.
 We denote by $\mathcal{H}_{\theta^{-1}}^0:=\mathcal{H}_{\theta^{-1}}\ominus \mathbb{C}$ the subspace of $\mathcal{H}_{\theta^{-1}}$ orthogonal to constants, that is, those with $f_0=0$. 
\end{definition}
With respect to the Fourier pairing
\[
\langle u,v\rangle=\sum_{k\in\mathbb Z}u_k\overline{v_k},
\]
the space $\mathcal H_{\theta^{-1}}$ is naturally identified with the Hilbert
dual of the space of analytic functions $\mathcal H_\theta$. Its elements may have Fourier coefficients growing
exponentially in $|k|$. The use of these types of spaces is not new in dynamics; see for example \cite{FR06}.

\begin{prop}\label{mult two}
For $0<\theta<1$, the operator $\mathcal{L} : \mathcal{H}_{\theta^{-1}} \to \mathcal{H}_{\theta^{-1}}$ is trace class. Moreover, $\mathcal{L}$ is a strict contraction when restricted to the subspace orthogonal to constants $\mathcal{L}: \mathcal{H}_{\theta^{-1}}^0\to \mathcal{H}_{\theta^{-1}}^0$ with
\[
\|\mathcal{L}\|_{\mathcal{H}^0_{\theta^{-1}}\to\mathcal{H}^0_{\theta^{-1}}}=\theta.
\]
\end{prop}

\begin{proof}
We consider the orthonormal basis $ e_k=\theta^{-|k|}z^k, k\in\mathbb{Z}$ of $\mathcal{H}_{\theta^{-1}}$ with respect to its norm  $\|f\|_{\theta^{-1}}:=\sqrt{\sum_{k\in\mathbb{Z}} |f_k|^2\theta^{2|k|}}$.

By definition,
\[
   \mathcal{L} e_k = \theta^{-|k|}\,z^{2k}=\theta^{|k|} e_{2k}.
\]
Therefore, the matrix of $\mathcal{L}$ in the orthonormal basis 
$\{e_k\}_{k\in\mathbb{Z}}$ has entries
\begin{equation}\label{matrix}
\mathcal{L}_{jk}:=\langle \mathcal{L}e_k,e_j\rangle_{\theta^{-1}}=\theta^{|k|}\delta_{j,2k}.
\end{equation}
As we can see, this is a very sparse matrix (with many zero elements).
Next, let us compute the singular values of $\mathcal{L}$, which by definition are the eigenvalues of $\sqrt{\mathcal{L}^*\mathcal{L}}$.
We find
\begin{align*}
\mathcal{L}^*_{jk} :&=\langle \mathcal{L}^*e_k, e_j\rangle_{\theta^{-1}} \\ &=\langle e_k,\mathcal{L}e_j\rangle_{\theta^{-1}} \\&=\overline{\langle  \mathcal{L}e_j,e_k\rangle_{\theta^{-1}}}\\&=\overline{\mathcal{L}_{kj}}\\&=\theta^{|j|}\delta_{k,2j}.
\end{align*}
Hence, 
\begin{align*}
(\mathcal{L}^*\mathcal{L})_{jk}&=\sum_{m\in\mathbb{Z}}\mathcal{L}^*_{jm}\mathcal{L}_{mk}\\&=\sum_{m\in\mathbb{Z}} \theta^{|j|}\delta_{m,2j}\theta^{|k|}\delta_{m,2k}\\&= \theta^{2|j|}\delta_{j,k}.
\end{align*}

Thus, \(\mathcal{L}^*\mathcal{L}\) is diagonal in the basis \(\{e_j\}\) with eigenvalues \(\theta^{2|j|}\). Therefore, \(\sqrt{\mathcal{L}^*\mathcal{L}}\) is diagonal in the basis \(\{e_k\}\) with eigenvalues \(\theta^{|j|}\). Hence, the singular values of $\mathcal{L}$ are
\[
   \sigma_j(\mathcal{L}) = \theta^{|j|}, \qquad j\in\mathbb{Z}_{\ge0},
\]
with multiplicity \(2\) for \(j>0\) (and multiplicity \(1\) for \(j=0\)).

Because \(0<\theta<1\),
\[
   \sum_{j\in\mathbb{Z}} \sigma_j(\mathcal{L})
     = 1 + 2\sum_{j\ge 1} \theta^j
     = 1 + \frac{2\theta}{1-\theta}
     < \infty,
\]
and hence $\mathcal{L}$ is trace class. 

The operator norm of $\mathcal{L}$ is equal to the largest singular value
\[
\|\mathcal{L}\|_{\mathcal{H}_{\theta^{-1}}\to\mathcal{H}_{\theta^{-1}}}=\max_{j\ge0} \sigma_j(\mathcal{L})=\sigma_0(\mathcal{L})=1.
\]
When restricted to the subspace orthogonal to constants $\mathcal{H}_{\theta^{-1}}^0$, we have
\[
\|\mathcal{L}\|_{\mathcal{H}^0_{\theta^{-1}}\to\mathcal{H}^0_{\theta^{-1}}}=\max_{j>0} \sigma_j(\mathcal{L})=\sigma_1(\mathcal{L})=\theta<1.
\]
Thus, $\mathcal{L}$ is a strict contraction on $\mathcal{H}_{\theta^{-1}}^0$.
\end{proof}

\begin{prop}
    The spectrum of $\mathcal{L}$ on $\mathcal{H}_{\theta^{-1}}$ 
    \[
    \sigma(\mathcal{L};\mathcal{H}_{\theta^{-1}})=\{0,1\},
    \]
    with $1$ being a simple eigenvalue corresponding to a constant eigenfunction.
    
    The spectrum of $\mathcal{L}$ restricted to its invariant subspace $\mathcal{H}^0_{\theta^{-1}}:=\mathcal{H}_{\theta^{-1}}\ominus\mathbb{C}$ is trivial
    \[
    \sigma(\mathcal{L};\mathcal{H}^0_{\theta^{-1}})=\{0\}.
    \]
\end{prop}

\begin{proof}
First, there is an $\mathcal{L}$-invariant decomposition
\[
\mathcal{H}_{\theta^{-1}}=\mathcal{H}^+_{\theta^{-1}}\oplus\mathbb{C}\oplus\mathcal{H}^-_{\theta^{-1}}, 
\]
where $\mathcal{H}^{\pm}_{\theta^{-1}}$ is the subspace generated by $e_k$ with $k>0$ ($k<0$ respectively). Note that $\mathcal{H}^+_{\theta^{-1}}\oplus\mathcal{H}^-_{\theta^{-1}}=\mathcal{H}^0_{\theta^{-1}}$.
Then
\[
\sigma(\mathcal{L}; \mathcal{H}_{\theta^{-1}})=\sigma(\mathcal{L}; \mathcal{H}^+_{\theta^{-1}})\cup \sigma(\mathcal{L}; \mathbb{C})\cup \sigma(\mathcal{L}; \mathcal{H}^-_{\theta^{-1}}).
\]
We have $\lambda=1$ as a simple eigenvalue with $e_0=1$ being the corresponding eigenfunction
\[
\sigma(\mathcal{L};\mathbb{C})=\{1\}.
\]

From \eqref{matrix}, we see that $\mathcal{L}$ restricted to $\mathcal{H}^+_{\theta^{-1}}$ or $\mathcal{H}^-_{\theta^{-1}}$ has an upper triangular structure with zeros on the diagonal. Since $\mathcal{L}$ is a trace class operator, it is compact, so is its restriction to a closed invariant subspace. If an operator has an upper triangular form and is compact, then its spectrum is equal to the set of matrix entries on the main diagonal together with $0$.

So,
\[
\sigma(\mathcal{L}; \mathcal{H}^+_{\theta^{-1}})=\sigma(\mathcal{L}; \mathcal{H}^-_{\theta^{-1}})=\{0\}.
\]
Thus,
\[
\sigma(\mathcal{L}; \mathcal{H}_{\theta^{-1}})=\{0,1\} \quad \text{and}\quad \sigma(\mathcal{L};\mathcal{H}^0_{\theta^{-1}})=\{0\}.
\]
\end{proof}

Operators with zero spectral radius are called \textit{quasi-nilpotent}. 

\begin{corollary}[Super exponential mixing]\label{super} Assume that $\mathcal{L}:\mathcal{H}^0_{\theta^{-1}}\to \mathcal{H}^0_{\theta^{-1}}$ is quasi-nilpotent. Then for any $\epsilon>0$, there exists $C>0$ such that for any $f\in \mathcal{H}^0_{\theta}$, $g\in\mathcal{H}_\theta$ and $n\ge0$
 \[
|\langle \mathcal{L}^n f,g\rangle_{L_2}| \le C\epsilon^n\|f\|_{{\theta}}\|g\|_{\theta}.
\]
\end{corollary}

\begin{proof}
 From the Gelfand formula for the spectral radius
 \[
r(\mathcal{L};\mathcal{H}^0_{\theta^{-1}})=\lim_{n\to\infty}\|\mathcal{L}^n\|^{1/n}_{\mathcal{H}^0_{\theta^{-1}}\to\mathcal{H}^0_{\theta^{-1}}},
 \]
 it follows that having a zero spectral radius is equivalent to super exponential decay of the norms of compositions of $\mathcal{L}$ with itself. In other words,
  for any $\epsilon>0$ there exists $C>0$ such that for any $n\ge0$
 \[
\|\mathcal{L}^n\|_{\mathcal{H}^0_{\theta^{-1}}\to\mathcal{H}^0_{\theta^{-1}}}\le C \epsilon^n.
 \]   
Applying the Cauchy-Schwarz duality inequality, it follows that for any $f\in \mathcal{H}^0_{\theta^{-1}}$ and $g\in\mathcal{H}_\theta$
 \[
|\langle \mathcal{L}^n f,g\rangle_{L_2}| \le \|\mathcal{L}^nf\|_{{\theta^{-1}}}\|g\|_{\theta} \le C\epsilon^n\|f\|_{\theta^{-1}}\|g\|_{\theta}.
 \]
 In particular, for any $f\in \mathcal{H}^0_{\theta}\subset \mathcal{H}_{\theta^{-1}}^0$, $g\in\mathcal{H}_\theta$ and any $n\ge0$
 \[
|\langle \mathcal{L}^n f,g\rangle_{L_2}| \le C\epsilon^n\|f\|_{{\theta}}\|g\|_{\theta}.\]
\end{proof}

\begin{rem} When $\theta=1$, the norms $\|\cdot\|_\theta$ and $\|\cdot\|_{\theta^{-1}}$ degenerate to the standard $L_2(\mathbb{S}^1,\mu)$-norm.
In this case, $\mathcal{L}$ is an isometry:
\[\|\mathcal{L}f\|_{L_2}=\|f\|_{L_2} \quad \text{for any}\quad f\in L_2(\mathbb{S}^1,\mu).\]
Therefore,
\[
   \mathcal{L}^* \mathcal{L} = \mathrm{Id}, \qquad \sigma_j(\mathcal{L})=1, \qquad j\in\mathbb{Z}.
\]

The condition $\theta<1$ is therefore essential for trace class and contraction properties.
\end{rem}

In fact, for any isometry $\mathcal{L}$, 
no decaying sequence can dominate all correlations of $L_2$ observables (see Corollary~\ref{no-uniform-mix}). That is why, in order to gain uniform control of the decay of correlations, spaces other than $L_2(\mathbb{S}^1,\mu)$ need to be considered. 

\section{Proof of Theorem~\ref{classification} in the Normalized Case}\label{proof}

\subsection{Rotations}
\begin{proof}[Proof of Part 1 of Theorem~\ref{classification}]
\label{rotations} Every $\mu$-preserving Blaschke product $\phi_j$ is represented by its Fourier series, converging everywhere in some neighborhood of the unit circle
\[
\phi_j(z)=\sum_{k=1}^\infty \frac{\phi_j^{(k)}(0)}{k!}z^k.
\]
Since $\phi_j$ takes the unit circle into itself, we have a restriction on its coefficients
\[
\sum_{k=1}^\infty \left|\frac{\phi_j^{(k)}(0)}{k!}\right|^2=\int_{\mathbb{S}^1}|\phi_j(z)|^2\,d\mu(z)=1.
\]
From here we see that the first derivative has to be bounded by 1
\[
|\phi_j'(0)|\le1.
\]
Moreover, if
\[
|\phi_j'(0)|=1,
\]
then all other coefficients have to vanish. Thus, $\phi_j$ is a rotation, i.e.,
\[
\phi_j(z)=\phi_j'(0) z=e^{i\psi}z, \quad \psi\in[0,2\pi),
\]
hence $\tau$ contains a rotation. 

Circle rotations are not strongly mixing (not even weakly, see \cite{HK95}). Thus we see that if for some $1\le j\le N$, $|\phi_j'(0)|=1$, then $\tau$ is not mixing. So, having $|\phi_j'(0)|<1$ for all $ 1\le j\le N$ is a necessary condition for $\tau$ to be  mixing. 
\end{proof}

\subsection{Koopman Precomposition Operator} For each $a\in\mathbb{F}_N^+$, the action $\tau(a)$ induces the \textit{Koopman} precomposition operator 
\begin{align*}
 \mathcal{L}^a:L_2(\mathbb{S}^1,\mu)&\to L_2(\mathbb{S}^1,\mu) \\
f&\mapsto f \circ  \tau(a).
\end{align*}
Note that the constant function $f(z)=c$ is invariant under $\mathcal{L}^a$
\[
\mathcal{L}^af=f\circ \tau(a)=f.
\]
Moreover, $\phi_1, \dots, \phi_N$ are assumed to be $\mu$-preserving, so $\mathcal{L}^a$ preserves orthogonality to the constants
\[
\int_{\mathbb{S}^1}\left(\mathcal{L}^af\right) \, d\mu=\int_{\mathbb{S}^1}\left(f\circ\tau(a)\right) \, d\mu=\int_{\mathbb{S}^1}f \, d\mu,
\]
thus inducing an operator 
\[
\mathcal{L}^a:L_2^0(\mathbb{S}^1,\mu)\to L_2^0(\mathbb{S}^1,\mu),
\]
where $L^0_2(\mathbb{S}^1,\mu):=L_2(\mathbb{S}^1,\mu)\ominus \mathbb{C}$, the subspace orthogonal to constants.

As in the case of the motivating example (see Section~\ref{motivating example}), $\mathcal{L}^a$ has good spectral properties on $\mathcal{H}_\theta^*=\mathcal{H}_{\theta^{-1}}$ (the space dual to analytic functions $\mathcal{H}_\theta$), as the next theorem shows.

\begin{theorem}\label{norm bound} Let $0<\theta<1$. Let $\phi_1, \dots, \phi_N\in\mathcal{B}^0$. 

\begin{enumerate}
    \item If $\phi_j'(0)=0$ for all $ 1\le j\le N$, then there exist $\beta, \gamma>0$ and $C>0$ such that for any $a\in \mathbb{F}_N^+$
    \[
    \|\mathcal{L}^a\|_{\mathcal{H}_{\theta^{-1}}^0\to \mathcal{H}_{\theta^{-1}}^0}\le C e^{-\beta e^{\gamma|a|}}.
    \]
    \item If $|\phi_j'(0)|<1$ for all $ 1\le j\le N$, then there exist $\kappa>0$ and $C>0$ such that for any $a\in \mathbb{F}_N^+$
    \[
    \|\mathcal{L}^a\|_{\mathcal{H}_{\theta^{-1}}^0\to \mathcal{H}_{\theta^{-1}}^0}\le C e^{-\kappa|a|}.
    \]
\end{enumerate}
\end{theorem}

The proof of this theorem is postponed until the next section. Now we use this result to prove sufficient conditions for exponential mixing and mixing at a double exponential rate on analytic observables in Theorem~\ref{classification}.
\begin{proof}[Proof of Part 2 and 3 of Theorem~\ref{classification}]
First, let us rewrite the correlation function in terms of the scalar product and the action of the precomposition operator $\mathcal{L}^a$. 
For any $f\in\mathcal{H}_\theta^0$ and $g\in\mathcal{H}_\theta$
\[
\int_{\mathbb{S}^1}\left(f\circ\tau(a)\right)\, \overline{g}\, d\mu =\langle \mathcal{L}^af, g\rangle_{L^2(\mathbb{S}^1,\mu)}.\]

 Assuming $f\in \mathcal{H}_\theta^0\subset \mathcal{H}_{\theta^{-1}}^0$, so  $\|f\|_{\theta^{-1}}\le \|f\|_{\theta}$, and applying the Cauchy-Schwarz inequality gives 
\begin{align*}
    |\langle \mathcal{L}^af, g\rangle_{L^2(\mathbb{S}^1,\mu)}|&\le \|\mathcal{L}^a f\|_{\theta^{-1}}\|g\|_{\theta} \\ &\le\|\mathcal{L}^a\|_{\mathcal{H}_{\theta^{-1}}^0 \to \mathcal{H}_{\theta^{-1}}^0}\|f\|_{\theta^{-1}}\|g\|_{\theta}
    \\ &\le \|\mathcal{L}^a\|_{\mathcal{H}_{\theta^{-1}}^0\to \mathcal{H}_{\theta^{-1}}^0}\|f\|_{\theta}\|g\|_{\theta}.
\end{align*}
The rest follows from Theorem~\ref{norm bound}.
\end{proof}

In the next subsection we prove that having $\phi_j'(0)=0$ for all $ 1\le j\le N$ is a necessary condition for $\tau$ to be mixing at a double exponential rate. 

\subsection{Lower Bound on Exponential Mixing}
In the next lemma we give a lower bound on the rate of exponential mixing in the case when $\phi_j'(0)\neq 0$ for some $ 1\le j\le N$. This way, we exclude mixing at a double exponential rate in this case  and finish the proof of Theorem~\ref{classification}.  

\begin{lemma}[Lower bound on exponential mixing]\label{lower bounds on exponential mixing} Let $\tau$ be as above. If $0\neq|\phi_j'(0)|<1$ for some $1\le j\le N$, then there exist constants $\kappa_0>0$, $C_0>0$, functions $f,g\in \mathcal{H}_\theta^0$%, $g\in\mathcal{H}_\theta$
, and a sequence $a_n\in\mathbb{F}_N^+, n\in\mathbb{N}$ with word length $|a_n|=n\to\infty$ such that
    \begin{equation}\label{lower bound for exponential m}
    \left| \int_{\mathbb{S}^1}\left(f\circ\tau(a_n)\right)\,\overline{g}\,d\mu \right|= C_0 e^{-\kappa_0|a_n|}\|f\|_{\theta}\|g\|_{\theta}.
    \end{equation}
\end{lemma}

\begin{proof} Let $\phi(z) :=\phi_j(z)$ and $\lambda:=\phi'_j(0)$, $0 < |\lambda| < 1$. Then
\[\phi(z)=\lambda z+\sum_{k=2}^\infty \frac{\phi^{(k)}(0)}{k!}z^k.
\]
    Consider $f(z)=g(z)=z$ and the sequence $a_n=j^n, n\in\mathbb{N}$, then $|a_n|=n$ and
    \[  \Big|\int_{\mathbb{S}^1}f(\tau(a)z)\overline{g(z)}d\mu (z)
\Big|=\Big|\frac{1}{2\pi i}\int_{\mathbb{S}^1}\phi^{n}(z)z^{-1}\frac{dz}{z}\Big|,
    \]
    where $\phi^{n}=\phi\circ \dots\circ \phi$ is $n$-fold composition of $\phi$ with itself. Next, we  apply the Residue Theorem to the last expression, use the chain rule, and $\phi(0)=0$, to get
    \begin{align*}
\Big|\int_{\mathbb{S}^1}f(\tau(a)z)\overline{g(z)}d\mu
\Big| &= | (\phi^{n})'(0) |\\  &= |\phi'(\phi^{n-1}(0))\dots\phi'(0)| \\&= |(\phi'(0))|^n\\&=|\lambda|^n.
    \end{align*}
   Using Definition~\ref{analytic space}, we have
    \[
    \|f\|_{\theta}=\|g\|_{\theta}=\theta^{-1}.
    \]
Now take $C_0=\left(\|f\|_{\theta}\|g\|_{\theta}\right)^{-1}=\theta^{2}$ and $\kappa_0=-\log |\lambda|>0$ to get \eqref{lower bound for exponential m}.      
\end{proof}
    
\section{Proof of Theorem~\ref{norm bound}}\label{one-more-proof}
First, we estimate the operator norm by its Hilbert-Schmidt norm
\[
  \|\mathcal{L}^a\|^2_{\mathcal{H}_{\theta^{-1}}^0\to \mathcal{H}_{\theta^{-1}}^0}\le\sum_{0\neq j,k\in\mathbb{Z}} |\langle\mathcal{L}^a e_k, e_j\rangle_{\mathcal{H}_{\theta^{-1}}}|^2%|(\mathcal{L}^a)_{j,k}|^2,
  \]  
  where %$(\mathcal{L}^a)_{j,k}=\langle\mathcal{L}^a e_k, e_j\rangle_{\mathcal{H}_\theta^*}$, and 
  $\{e_k=\theta^{-|k|}z^k\}_{k\in\mathbb{Z}}$ is an orthonormal basis of $\mathcal{H}_{\theta^{-1}}$ (since $\|z^k\|_{\theta^{-1}}=\theta^{|k|}$). Note that in the sum we take only non-zero $j$ and $k$, since we restrict the operator to the subspace $\mathcal{H}_{\theta^{-1}}^0$ orthogonal to constant functions. 

\begin{lemma}\label{main formula} For $a\in\mathbb{F}_N^+$ and $j,k\in\mathbb{Z}$,
 \begin{align*}
  |\langle\mathcal{L}^a e_k, e_j\rangle_{\mathcal{H}_{\theta^{-1}}}|= \theta^{|j|-|k|}\left|\frac{1}{2\pi i}\int_{\mathbb{S}^1}(\tau(a)z)^k z^{-j}\frac{dz}{z}\right|,
 \end{align*}
 and
 \[\left|\frac{1}{2\pi i}\int_{\mathbb{S}^1}(\tau(a)z)^k z^{-j}\frac{dz}{z}\right|\le 1.\]
\end{lemma}

\begin{proof}
We use $e_k=\theta^{-|k|}z^k$ and the definition of the scalar product in $\mathcal{H}_{\theta^{-1}}$ %and Riesz isomorphism lemma
\begin{align*}
|\langle\mathcal{L}^a e_k, e_j\rangle_{\mathcal{H}_{\theta^{-1}}}|&=\theta^{-|k|-|j|}|\langle \mathcal{L}^a z^k, z^j\rangle_{\mathcal{H}_{\theta^{-1}}}|\\&=\theta^{-|k|-|j|}\theta^{2|j|}|\langle \mathcal{L}^a z^k, z^j\rangle_{L_2}|\\ &=\theta^{|j|-|k|}\left|\frac{1}{2\pi i}\int_{\mathbb{S}^1}(\mathcal{L}^a z^k) z^{-j}\frac{dz}{z}\right|\\ &=\theta^{|j|-|k|}\left|\frac{1}{2\pi i}\int_{\mathbb{S}^1}(\tau(a)z)^k z^{-j}\frac{dz}{z}\right|.
\end{align*}
Since $|\tau(a)z|=1$ for $|z|=1$,
\[
\left|\frac{1}{2\pi i}\int_{\mathbb{S}^1}(\tau(a)z)^k z^{-j}\frac{dz}{z}\right|\le 1.
\]
\end{proof} 
\begin{corollary}\label{bounds using symmetry}
For $a\in\mathbb{F}_N^+$, 
\[
\|\mathcal{L}^a\|^2_{\mathcal{H}_{\theta^{-1}}^0\to \mathcal{H}_{\theta^{-1}}^0}\le2\sum_{\substack{k>0 \\ 0\neq j\in\mathbb{Z}}} \theta^{2(|j|-|k|)}\left|\frac{1}{2\pi i}\int_{\mathbb{S}^1}(\tau(a)z)^k z^{-j}\frac{dz}{z}\right|^2.
\]  
\end{corollary}
\begin{proof}
    First, let us notice that since $|\tau(a)z|=1$ for $|z|=1$, 
\[
\frac{1}{2\pi i}\int_{\mathbb{S}^1}(\tau(a)z)^k z^{-j}\frac{dz}{z}=\overline{\frac{1}{2\pi i}\int_{\mathbb{S}^1}(\tau(a)z)^{-k} z^{j}\frac{dz}{z}},
\]
so by Lemma~\ref{main formula},
\begin{equation*}
 %|(\mathcal{L}^a)_{j,k}|=|(\mathcal{L}^a)_{-j,-k}|
 |\langle\mathcal{L}^a e_k, e_j\rangle_{\mathcal{H}_{\theta^{-1}}}|=|\langle\mathcal{L}^a e_{-k}, e_{-j}\rangle_{\mathcal{H}_{\theta^{-1}}}|.
\end{equation*}
Therefore,
\[
\|\mathcal{L}^a\|^2_{\mathcal{H}_{\theta^{-1}}^0\to \mathcal{H}_{\theta^{-1}}^0}\le2\sum_{\substack{k>0 \\ 0\neq j\in\mathbb{Z}}} |\langle\mathcal{L}^a e_k, e_j\rangle_{\mathcal{H}_{\theta^{-1}}}|^2.
\]
Using Lemma~\ref{main formula}, we finish the proof of the corollary. 
\end{proof}

Next, we show that many of the coefficients in the sum are zero. This allows us to obtain double exponentially decaying upper estimates from Theorem~\ref{norm bound}.

\subsection{Double Exponential Rate}

\begin{lemma}\label{vanishing coefficients}
    Let $\phi_m'(0)=0$ for all $ 1\le m\le N$. Let $p$ be the minimal local degree of the functions $\phi_m$, $ 1\le m\le N$, at $z=0$. Then for any $a\in \mathbb{F}_N^+$, $k>0$, and $j<p^{|a|}k$ %or $\{j>p^{|a|}k, k<0\}$
    \[
    \left|\frac{1}{2\pi i}\int_{\mathbb{S}^1}(\tau(a)z)^k z^{-j}\frac{dz}{z}\right|=0.
    \] 
\end{lemma}

\begin{proof} 
Since $\phi_m'(0)=0$, $ 1\le m\le N$, it follows that $p\ge 2$ and
  \[
  \phi_m(z)=\sum_{k=p}^\infty \frac{\phi_m^{(k)}(0)}{k!}z^k=O(z^p).
  \]
Hence, an $n$-fold composition of $\phi_m$ with itself can be estimated as \[
\phi_m^{n}(z)=O(z^{p^n}).
\] 
Similarly for any $a\in\mathbb{F}_N^+$ the action $\tau(a)$ can be estimated as \[\tau(a)z=O(z^{p^{|a|}}).\]
Then for any $j<p^{|a|}k$
\[
(\tau(a)z)^kz^{-j-1}=O(1)
\]
is a function analytic in the unit disk. By Cauchy Theorem, its integral must be zero.
\end{proof}
\begin{proof}[Proof of Part 1 of Theorem~\ref{norm bound}]
From  Lemma~\ref{main formula}, Corollary~\ref{bounds using symmetry}, and Lemma~\ref{vanishing coefficients} it follows that for $a\in\mathbb{F}_N^+$
\[
\|\mathcal{L}^a\|^2_{\mathcal{H}_{\theta^{-1}}^0\to \mathcal{H}_{\theta^{-1}}^0}
\le 2\sum_{\substack{k>0 \\ p^{|a|}k\le j}} \theta^{2(|j|-|k|)}.
\]
Now, applying the geometric series formula twice, we get for $|a|>0$ 

\begin{align*}
\|\mathcal{L}^a\|^2_{\mathcal{H}_{\theta^{-1}}^0\to \mathcal{H}_{\theta^{-1}}^0}&\le 2\sum_{k>0}\frac{{\theta^{2(p^{|a|}-1)k}} }{1-\theta^2}\\&=\frac{2\theta^{2(p^{|a|}-1)}}{(1-\theta^2)(1-\theta^{2(p^{|a|}-1)})} \\&\le \frac{2\theta^{2(p^{|a|}-1)}}{(1-\theta^{2})^2}.
\end{align*}
Therefore
    \[
   \|\mathcal{L}^a\|_{\mathcal{H}_{\theta^{-1}}^0\to \mathcal{H}_{\theta^{-1}}^0}\le C e^{-\beta e^{\gamma|a|}},
    \]
    where \begin{equation}\label{constants for double exponential mixing}
    C=\frac{\sqrt{2}}{\theta(1-\theta^2)}>0, \quad  \beta=-\log{\theta}>0, \quad \gamma=\log{p}\ge \log 2,
    \end{equation}
    with $p$ as in Lemma~\ref{vanishing coefficients}.
\end{proof}

\subsection{Exponential Rate}\label{Exponential mixing via complex analysis}
The next lemma is based on the application of the Schwarz lemma to  the case of free semigroup action $\mathbb{F}_N^+$ (cf. \cite{ABCC22}, Lemma 3.3).

\begin{lemma}\label{Schwarz lemma} Let $|\phi_m'(0)|<1$ for all  $1\le m\le N$. 
   Then for any $0<r<1$, there exists $\lambda=\lambda(r)<1$, function continuous in $r$, such that for any $a\in\mathbb{F}_N^+$
   \[|\tau(a)z|\le \lambda^{|a|}r, \quad |z|=r.
   \]
\end{lemma}

\begin{proof}
    For each $1\le m\le N$, let $M_m(r)$ be the maximum of $|\phi_m(z)|$ for $|z|\le r$. Since $|\phi_m'(0)|<1$, by the Schwarz lemma, 
    \[
    \lambda_m(r):=\frac{M_m(r)}{r}<1. 
    \]
    Then we apply the Schwarz lemma again to the function 
    \[
    \psi_m(z)=\frac{\phi_m(rz)}{M_m(r)}, \quad \ |z|\le 1,
    \]
    and we get 
\[
    |\psi_m(z)|\le |z| \quad \text{for} \quad |z|\le 1.
    \]
It follows that 
\[
|\phi_m(z)|\le \frac{M_m(r)}{r}|z|=\lambda_m(r)|z| \quad \text{for} \quad |z|\le r.
\]
Now let \[\lambda=\lambda(r):=\max_{1\le m\le N}\lambda_m(r)<1,\]
so for each $1\le m\le N$
\[
|\phi_m(z)|\le \lambda|z|\quad \text{for} \quad |z|\le r.
\]
Hence, for
$a=i_1i_2\dots i_n$ with $|a|=n$ and $|z|\le r$,
\begin{align*}
|\tau(a)z|&=\bigl|\phi_{i_n}\circ\dots\circ\phi_{i_1}(z)\bigr|
\\&\le\lambda\bigl|\phi_{i_{n-1}}\circ\dots\circ\phi_{i_1}(z)\bigr|\\
&\le\dots\\&\le\lambda^{n}|z| .
\end{align*}
Taking $|z|=r$ gives the claim.
\end{proof}

\begin{corollary}\label{corollary from Schwarz lemma} Let $|\phi_m'(0)|<1$ for all $1\le m\le N$. Then for any $0<r<1$, there exists $\lambda=\lambda(r)<1$ such that for any $a\in\mathbb{F}_N^+$ and $k>0$, $j\in\mathbb{Z}$
    \[
   \left|\frac{1}{2\pi i}\int_{\mathbb{S}^1}(\tau(a)z)^k z^{-j}\frac{dz}{z}\right|\le \lambda^{|a|k}r^{k- j}.
    \]
\end{corollary}

\begin{proof}
    We move the contour of integration from the unit circle to $|z|=r$
    \[
   \left|\frac{1}{2\pi i}\int_{\mathbb{S}^1}(\tau(a)z)^k z^{-j}\frac{dz}{z}\right|=\left|\frac{1}{2\pi i}\int_{|z|=r}(\tau(a)z)^k z^{-j}\frac{dz}{z}\right| \le \lambda^{|a|k}r^{k-j}.
    \]
\end{proof}

\begin{proof}[Proof of Part 2 of Theorem~\ref{norm bound}]
 From  Corollary~\ref{bounds using symmetry} and Corollary~\ref{corollary from Schwarz lemma}, it follows that for any $0<\theta<1$ and any $1<r<1$ there exists $\lambda=\lambda(r)<1$ such that
\begin{align*}
\|\mathcal{L}^a\|^2_{\mathcal{H}_{\theta^{-1}}^0\to \mathcal{H}_{\theta^{-1}}^0}&
\le 2\sum_{k>0, j\neq0}\theta^{2 (-|k|+|j|)} \lambda^{2|a|k}r^{2(k-j)} \\&\le  2 \sum_{k>0} \theta^{-2 k} r^{2k}\lambda^{2|a|k}\,\cdot \sum_{j\neq0} \theta^{2|j|} r^{-2j}.
\end{align*}
Given any $0<\theta<1$, we choose the radius $r$. By Lemma~\ref{Schwarz lemma} we have $\theta\lambda(\theta)<\theta$,
and $r\mapsto r\lambda(r)$ is continuous, so we may fix
\begin{equation}\label{eq:choice of r}
r\in(\theta,1)\qquad\text{with}\qquad r\lambda(r)<\theta .
\end{equation}
Since $\lambda^{|a|}\le \lambda$ for any $|a|\ge 1$, we have $r\lambda^{|a|}< \theta$ for any $|a|\ge 1$. Thus, both series converge for any $|a|\ge1$. Applying the geometric series formula twice, we get for $|a|>0$ 
\begin{align*}
\|\mathcal{L}^a\|^2_{\mathcal{H}_{\theta^{-1}}^0\to \mathcal{H}_{\theta^{-1}}^0}&
\le 2 \cdot\frac{\theta^2}{\theta^2-r^2\lambda^{2|a|}}\,\cdot\frac{r^{2}\lambda^{2|a|}}{\theta^{2}}\cdot \left(\frac{\theta^{2}}{r^{2}-\theta^{2}}+\frac{\theta^{2}r^{2}}{1-\theta^{2}r^{2}}\right)\\&\le \frac{2\theta^2r^{2}}{\theta^2-r^2\lambda^{2}}\,\cdot \left(\frac{1}{r^{2}-\theta^{2}}+\frac{r^{2}}{1-\theta^{2}r^{2}}\right)\cdot \lambda^{2|a|}
\end{align*}
Therefore, for every $a\in\mathbb{F}_N^+$,
    \[
   \|\mathcal{L}^a\|_{\mathcal{H}_{\theta^{-1}}^0\to \mathcal{H}_{\theta^{-1}}^0}\le C e^{-\kappa|a|},
    \]
    where
    \begin{equation}\label{constants for exponential mixing}
    C=\sqrt{\frac{2r^{2}\theta^2}{\theta^2-r^2\lambda^{2}}\left(\frac{1}{r^{2}-\theta^{2}}+\frac{r^{2}}{1-\theta^{2}r^{2}}\right)}>0,
    \qquad \kappa=-\log{\lambda}>0,
    \end{equation}
    with $r$ as in \eqref{eq:choice of r} and $\lambda=\lambda(r)$ as in Lemma~\ref{Schwarz lemma}.
\end{proof}

\section{Sharpness of the Exponent
\texorpdfstring{$\gamma=\log p$}{gamma = log p}}\label{sharp-and-rigid}

Following our model example in Proposition~\ref{multiplication-by-2-is-DEM}, we prove the following statement.
\begin{theorem}[Sharpness of the exponent in double exponential rate]\label{thm:sharpness}
Let $\phi\in\mathcal{B}^0$ have local degree $p\ge2$ at the origin. There exist
$f,g\in\mathcal{H}_\theta$ and constants $C'>0$, $\beta'>0$ such that
\begin{equation}\label{eq:headline}
\mathcal{C}_{f,g}(n)\ \ge\ C'e^{-\beta' e^{\gamma n}}\|f\|_\theta\|g\|_\theta,
\qquad \gamma=\log p,
\end{equation}
for infinitely many $n$. 
\end{theorem}
\begin{proof} If $\phi$ has the local degree $p$ at $z=0$, then $\phi^{(p)}(0) \neq 0$ and \begin{align*}
  \phi(z)&=\sum_{k=p}^\infty \frac{\phi^{(k)}(0)}{k!}z^k\\&=c_pz^p+O(z^{p+1}),
  \end{align*}
  where \[c_p:=\frac{\phi^{(p)}(0)}{p!}.\]
Its $n$-fold composition with itself is then \begin{align}\label{fourier-coef-of-iteration}
\phi^{n}(z)&=c_p^{1+p+p^2+\dots+p^{n-1}} z^{p^n}+O(z^{p^n+1})\nonumber\\&=c_p^{\frac{p^n-1}{p-1}} z^{p^n}+O(z^{p^n+1}).
\end{align}
Let $\theta_*$ be defined as 
\begin{equation}\label{eq:theta_*}
\theta_*:=\theta\,|c_p|^{\frac{1}{p-1}}.
\end{equation}
Because $|\phi|=1$ on $\mathbb S^1$, Parseval gives
\[
\sum_{k\ge p}|c_k|^2=1.
\]
Hence $0<|c_p|\le1$, and therefore
\[
0<\theta_*:=\theta|c_p|^{1/(p-1)}\le\theta<1.
\]
Choose $M\in\mathbb{N}$ so that
\begin{equation}\label{eq:choice of M}
\theta^{\,p^{M}}<\theta_* ,
\end{equation}
which is possible since $\theta<1$ and $\theta_*>0$, and set
\[
f(z)=z,
\qquad
g(z)=\sum_{l\ge1}\frac{1}{l}\,\theta^{\,p^{lM}}z^{\,p^{lM}}.
\]
Then $f,g\in\mathcal{H}_\theta$ with \[\|f\|_\theta=\theta^{-1} \quad \text{and}\quad
\|g\|_\theta=\sqrt{\sum_{l\ge1} \frac{1}{l^2}}=\pi/\sqrt6.\]

The series defining $g$ is lacunary: it is supported on the sparse set of
frequencies $\{p^{lM}\}_{l\ge1}$. For the Bernoulli map $\phi(z)=z^{2}$ from Proposition~\ref{multiplication-by-2-is-DEM}, precomposition sends a character to a
character, so in~\eqref{one character survives} only one mode survives in the correlation function. In the nonlinear case $f\circ\phi^{\,n}$ spreads over infinitely
many modes, so we choose $g$ lacunary to make the tail negligible compared to
the leading term. 

Now fix $n\in M\mathbb{N}$ and let $i:=n/M$. Since $\int f\,d\mu=0$,
\[
\mathcal{C}_{f,g}(n)=\Bigl|\sum_{l\ge1}\phi^{\,n}_{p^{lM}} \,\overline{g_{p^{lM}}}\Bigr| .
\]
We treat the three ranges of $l$ separately.

If $l<i$ then $p^{lM}<p^{n}$, and $\phi^{\,n}(z)=O(z^{p^n})$, so $\phi^{\,n}_{p^{lM}}=0$.

If $l=i$ the term equals
\[
c_p^{\frac{p^n-1}{p-1}}\frac{1}{i}\theta^{\,p^{n}}
=c_p^{\frac{p^n-1}{p-1}}\,\theta^{\,p^{n}}\frac{M}{n}.
\]

If $l>i$, then by the Cauchy--Schwarz inequality and $\|\phi^{\,n}\|_{L_2}=1$,
\[
\Bigl|\sum_{l>i}\phi^{\,n}_{p^{lM}}\,\overline{g_{p^{lM}}}\Bigr|
\le\sqrt{\sum_{l>i}\bigl|\phi^{\,n}_{p^{lM}}\bigr|^2}
\sqrt{\sum_{l>i}\frac{\theta^{2p^{lM}}}{l^2}}
\le \sqrt{\frac{M}{n}}\theta^{\,p^{\,n+M}} .
\]
Therefore, using \eqref{eq:theta_*}, we get
\begin{equation}\label{eq:main minus tail}
\mathcal{C}_{f,g}(n)\ \ge\ \frac{M}{n}\,|c_p|^{-\frac{1}{p-1}}\,\theta_*^{\,p^n}\ -\ \sqrt{\frac{M}{n}}\bigl(\theta^{\,p^{M}}\bigr)^{p^n}.
\end{equation}

Both terms in \eqref{eq:main minus tail} are of the form (base)$^{p^n}$, up to a power
factor in $n$. By the choice of $M$
\eqref{eq:choice of M} we have $\theta^{\,p^{M}}<\theta_*$, hence
\[
\frac{\bigl(\theta^{\,p^{M}}\bigr)^{p^n}}{\theta_*^{\,p^n}}
=\Bigl(\frac{\theta^{\,p^{M}}}{\theta_*}\Bigr)^{p^n}\longrightarrow0
\]
double exponentially, which dominates the factor $\sqrt{n/M}$. Consequently, the tail in
\eqref{eq:main minus tail} is at most half the main term for all large $n$, and
\[
\mathcal{C}_{f,g}(n)\ \ge\ \frac{M}{2n}\,|c_p|^{-\frac{1}{p-1}}\,\theta_*^{\,p^n} 
\]
for all large $n\in M\mathbb{N}$.
Finally, the factor $M/n$ is absorbed into any $\beta'>\beta_*=\log\frac1{\theta_*}$,
since $\frac{M}{n}\theta_*^{\,p^n}\ge e^{-\beta'p^n}$ for all large $n$. Recalling
$p^n=e^{\gamma n}$ completes the proof.
\end{proof}

\begin{corollary}[Rigidity in fixed degree]\label{cor:rigidity fixed degree}
Assume that $\phi\in\mathcal{B}^0$ is a degree $q\ge2$ map and that it is mixing
at a double exponential rate with $\gamma=\log q$. Then $\phi(z)=e^{i\psi}z^{q}$
for some $\psi\in[0,2\pi)$.
\end{corollary}

\begin{proof}
Let $p$ be the local degree of $\phi$ at the origin. We have $p\le q$. By
\eqref{constants for double exponential mixing} the exponent $\log p$ is
attained, and by Theorem~\ref{thm:sharpness} no larger exponent is.
Since the hypothesis says that $\log q$ is attained, $q\le p$. Thus, the local degree equals the topological degree, so all $q$ zeros of
$\phi$ lie at the origin and $\phi(z)=e^{i\psi}z^{q}$.
\end{proof}

\section{Failure of \texorpdfstring{$C^1$}{C1}-stability of Mixing at Double Exponential Rate}\label{examples} 
The degree two maps $\phi\in\mathcal{B}^0$ are of the form
\[
\phi(z)=e^{i\psi}z\,\frac{z-\lambda}{1-\bar{\lambda}z},
\qquad |\lambda|<1,\ \psi\in[0,2\pi),
\]
so that $|\phi'(0)|=|\lambda|$. By
Corollary~\ref{intro:classification on S1}, such a map is mixing at a double
exponential rate if and only if $\lambda=0$, that is, if and only if it is the
\textit{affine} map $\phi(z)=e^{i\psi}z^{2}$; this is the smallest case of
Corollary~\ref{cor:rigidity fixed degree}. In degree two the double exponential rate is
therefore rigid: a map in $\mathcal{B}^0$ that is close to $z^{2}$ has the same
degree, hence is again of the above form, so the only nearby maps retaining the rate are
the affine maps $e^{i\psi}z^{2}$ themselves.

In degree three the picture is different, and both behaviors occur arbitrarily close to the affine model, as the following example shows.

\begin{prop}\label{1d examples}
For $0< |\lambda|<1$, let \vspace{-3mm}\[
\phi_1(z):=z^2\frac{z-\lambda}{1-\overline{\lambda} z} \quad \text{and} \quad \phi_2(z):=z\left(\frac{z-\lambda}{1-\overline{\lambda} z}\right)^2. 
\]
Then $\phi_1$ is mixing at a double exponential rate, but  $\phi_2$ is not. 
For small $\lambda$, both $\phi_1$ and $\phi_2$ are $C^1$-perturbations of \[\phi_0:z\mapsto z^3, \] that is, $d_{C^1}(\phi_{i}, \phi_0)\le C |\lambda|$ for both $i=1,2$. 
However, neither $\phi_1$ nor $\phi_2$ is $C^1$-conjugate to $\phi_0$. 
\end{prop}

These examples show that mixing at a double exponential rate is \textit{not a $C^1$-stable} property, since any $C^1$-neighborhood of $z^3$ contains maps of both types: those that are mixing at a double exponential rate and those that are not. Due to the structural stability (see \cite{Shub87}), there is a $C^1$-neighborhood of $z^3$ in which every Lebesgue measure preserving $C^{1+\alpha}$ map is exponentially mixing again.

Moreover, Proposition~\ref{1d examples} shows that there exist \textit{nonlinear} maps that mix at a double exponential rate and are not $C^1$-conjugate to affine maps. In other words, if the exponent is not prescribed, the rate is not rigid: it determines nothing beyond the local degree at the
fixed point. Rigidity is recovered exactly when the exponent is required to be
maximal for the degree, which is
Corollary~\ref{cor:rigidity fixed degree}.

\begin{proof}
  By Corollary~\ref{intro:classification on S1}, for non-zero $\lambda$, $\phi_1$ is mixing at a double exponential rate but $\phi_2$ is not. Let us note that both are exponentially mixing on H\"older observables since they are uniformly expanding maps.

  By taking $\lambda\to 0$, we can see that both $\phi_1$ and $\phi_2$ are one-parameter deformations of the circle endomorphism \[\phi_0: z\mapsto z^3,\]
  which is mixing at a double exponential rate.

  In fact, for small $|\lambda|$, $\phi_1$ and $\phi_2$ can be viewed as $C^1$-perturbations of $\phi_0$. We will show that there exists $C>0$ such that
\begin{equation}\label{C^1 estimates}
d_{C^1}(\phi_{i},\phi_0) := \sup_{z\in\mathbb{S}^1} |\phi_i(z)-\phi_0(z)| + \sup_{z\in\mathbb{S}^1} |\phi_i'(z) - \phi_0'(z)|\le C|\lambda|
\end{equation}
for any $i=1,2$.
Since
\begin{align*}
\phi_1(z)&=z^2 (z-\lambda)\left(1 + \overline{\lambda} z + O(|\lambda|^2)\right) \\&= z^3 - \lambda z^2 + \overline{\lambda} z^4 + O(|\lambda|^2),
\end{align*}
then
\[
\sup_{z\in\mathbb{S}^1}|\phi_1(z) - \phi_0(z)| =\sup_{z\in\mathbb{S}^1}| - \lambda z^2 + \overline{\lambda} z^4 + O(|\lambda|^2)|\le C_0|\lambda|.
\]
Next, we compute the derivative 
\begin{align*}
\phi_1'(z) &= 2 z \frac{z-\lambda}{1-\overline{\lambda} z} + z^2 \frac{1 - |\lambda|^2}{(1-\overline{\lambda} z)^2}\\&=2 z (z-\lambda)(1 + \overline{\lambda} z + O(|\lambda|^2)) + z^2 (1 - |\lambda|^2)(1 + 2 \overline{\lambda} z + O(|\lambda|^2))\\&=
3z^2-2\lambda z+4\bar{\lambda} z^3+O(|\lambda|^2),
\end{align*}
so
\begin{align*}
\sup_{z\in\mathbb{S}^1} |\phi_1'(z) - \phi_0'(z)| = \sup_{z\in\mathbb{S}^1}|- 2 \lambda z + 4 \overline{\lambda} z^3 + O(|\lambda|^2)|
\le C_1|\lambda|.
\end{align*}
We get similar estimates for $\phi_2$, thus proving \eqref{C^1 estimates}. In fact, they are $C^{k}$-perturbations of $\phi_{0}$ for any $k$,
$C^{\infty}$-perturbations, and $C^{\omega}$-perturbations.

Now we are going to prove that $\phi_1$ and $\phi_2$ are not $C^1$-conjugate to $z^3$. In fact, we are proving more: they are not even absolutely continuously conjugate. Note that $\phi_{1}$ and $\phi_2$ are topologically conjugate to $\phi_0$ (see \cite{Shub69}); that is, there exist homeomorphisms $h_i:\mathbb{S}^1\to\mathbb{S}^1$ such that
  \[
  \phi_{i}\circ h_i=h_i\circ \phi_0
  \]
  for any $i=1,2$.
  From the work of Shub and Sullivan \cite{SS85}, it follows that if two finite Blaschke products are absolutely continuously conjugate, then they are conjugate by M\"obius transformation
  \[h_i(z)=e^{i\psi_i} \frac{z-a_i}{1-\overline{a_i}z}\]
  for some $|a_i|<1$ and $ \psi_i\in[0,2\pi), i=1,2.$ 
  %theta?
  It is a conformal map, and composition with a conformal map preserves local degree. Note that local degrees of $\phi_0,\phi_1,\phi_2$ at $z=0$ are $3$, $2$, and $1$ respectively. Thus, neither $\phi_1$ nor $\phi_2$ is $C^1$-conjugate to $\phi_0$.
\end{proof}
This example can be immediately generalized to the following statement.

\begin{prop}\label{nowhere dense}
Within the space of Lebesgue measure preserving Blaschke products $\mathcal{B}^0$, the ones that are mixing at a double exponential rate form a nowhere dense set in the $C^1$-topology on $\mathbb{S}^1$.
\end{prop}
\begin{proof} First, within the space of Lebesgue measure preserving Blaschke products $\mathcal{B}^0$, the ones that are mixing at a double exponential rate  form a closed set, let us denote it by $\mathcal{D}$. It follows from the fact that $\mathcal{D}$ is a kernel of a continuous linear functional, which can be seen from the Cauchy formula
\[
\phi'(0)=\frac{1}{2\pi i} \int_{\mathbb{S}^1} \phi(z)\frac{dz}{z^2}=\int_{\mathbb{S}^1} \phi(z) \, \overline{z}\, d \mu=\langle \phi,e_1\rangle_{L_2},
\]
where $e_1(z)=z$. Next, we show that for any $\phi\in\mathcal{D}$, any of its $C^1$-neighborhoods will contain a map from $\mathcal{B}^0\setminus\mathcal{D}$. 
    Any map $\phi\in\mathcal{D}$ mixing at a double exponential rate  has the form \[
    \phi(z)=e^{i\psi}z\prod_{i=1}^q\frac{z-\lambda_i}{1-\overline{\lambda_i}z},\]
    where $\psi\in[0,2\pi),\lambda_1=0, |\lambda_i|<1, i=1,\dots,q,\, q\in\mathbb{Z}_{\ge1}$. Now, deform each factor of $z$ into a non-trivial Blaschke factor $\frac{z-\lambda}{1-\overline{\lambda} z}$, except for one. Thus, the new map $\phi_\lambda(0)=0$ still preserves the Lebesgue measure. On the other hand, $\phi_\lambda'(0)\neq 0$, so it is not mixing at a double exponential rate anymore. By taking $\lambda$ small enough, we make sure that $\phi_\lambda$ is in the $C^1$-neighborhood of $\phi$. 
\end{proof}

 In~\cite{BN19}, Bandtlow and Naud use Blaschke products to prove a stronger result: there exists a dense set of analytic expanding maps of the circle with a nontrivial spectrum, which implies that these maps are not mixing at a double exponential rate.

%% file: quasi-nilpotent.tex
\section{Spectrum, Quasi-Nilpotency, and Resonances}\label{spec-qnilp-res}

\begin{theorem}\label{thm:spectrum}
Let $\phi\in\mathcal{B}^0$ and $\lambda:=\phi'(0)$, and assume that $|\lambda|<1$. Then for every
$0<\theta<1$, the spectrum of $\mathcal{L}$ on $\mathcal{H}_{\theta^{-1}}$ is
\[
\sigma(\mathcal{L};\mathcal{H}_{\theta^{-1}})
=\{0,1\}\cup\bigl\{\lambda^{m},\ \overline{\lambda}^{\,m}\ :\ m\ge1\bigr\},
\]
with $1$ a simple eigenvalue corresponding to a constant eigenfunction, and on
the invariant subspace
$\mathcal{H}^{0}_{\theta^{-1}}=\mathcal{H}_{\theta^{-1}}\ominus\mathbb{C}$
\[
\sigma(\mathcal{L};\mathcal{H}^{0}_{\theta^{-1}})
=\{0\}\cup\bigl\{\lambda^{m},\ \overline{\lambda}^{\,m}\ :\ m\ge1\bigr\}.
\]
If $\lambda\neq0$, the eigenvalue $\lambda^{m}$ has eigenfunction $\chi(z)^{\,m}$,
where $\chi$ is the K\"onigs function of $\phi$, and
$\overline{\lambda}^{\,m}$ has eigenfunction $\overline{\chi(z)}^{\,m}$. These
eigenfunctions belong to $\mathcal{H}_{\theta^{-1}}$ for \emph{every}
$0<\theta<1$, while none of them belongs to $L_2(\mathbb{S}^1,\mu)$.
\end{theorem}

The description of the spectrum is due to Bandtlow, Just, and Slipantschuk \cite{BJS17}. 

\begin{proof}
First, there is an $\mathcal{L}$-invariant decomposition
\[
\mathcal{H}_{\theta^{-1}}=\mathcal{H}^{+}_{\theta^{-1}}\oplus\mathbb{C}\oplus\mathcal{H}^{-}_{\theta^{-1}},
\]
where $\mathcal{H}^{\pm}_{\theta^{-1}}$ is the subspace generated by $e_k$ with
$k>0$ ($k<0$ respectively): indeed $\mathcal{L}e_0=1$, while
$\mathcal{L}e_k=(\phi(z))^{k}$ is analytic and vanishes at the origin for $k>0$, and
$\mathcal{L}e_{-k}=\overline{\phi(z)^{k}}$ for $k>0$ because $|\phi|=1$ on
$\mathbb{S}^1$. Hence
\[
\sigma(\mathcal{L};\mathcal{H}_{\theta^{-1}})=\sigma(\mathcal{L};\mathcal{H}^{+}_{\theta^{-1}})
\cup\sigma(\mathcal{L};\mathbb{C})\cup\sigma(\mathcal{L};\mathcal{H}^{-}_{\theta^{-1}}),
\]
and $\sigma(\mathcal{L};\mathbb{C})=\{1\}$, with eigenfunction $e_0=1$.

The matrix of $\mathcal{L}$ restricted to $\mathcal{H}^{+}_{\theta^{-1}}$
is upper triangular, with diagonal entries
\[
\mathcal{L}_{kk}=\phi'(0)^{k}=\lambda^{k},\qquad k\ge1 .
\]
In the proof of Theorem~\ref{norm bound}, we in fact bound the Hilbert--Schmidt
norm of $\mathcal{L}$, so $\mathcal{L}$ is compact, and so is its restriction to
a closed invariant subspace. If a compact operator has upper triangular form,
its spectrum is the set of diagonal entries together with $0$. Therefore
\[
\sigma(\mathcal{L};\mathcal{H}^{+}_{\theta^{-1}})=\{0\}\cup\{\lambda^{m}:m\ge1\},
\qquad
\sigma(\mathcal{L};\mathcal{H}^{-}_{\theta^{-1}})=\{0\}\cup\{\overline{\lambda}^{\,m}:m\ge1\}.
\]

Let now $\lambda\neq0$. By K\"onigs theorem \cite[Ch.~6]{Sh93} the limit
\[\chi=\lim_{n\to\infty}\phi^{\,n}/\lambda^{n}\] exists locally uniformly on
$\mathbb{D}$ and satisfies \[\chi\circ\phi=\lambda\chi.\]%$\chi(0)=0$,$\chi'(0)=1$.
Consequently
$\mathcal{L}(\chi(z)^{m})=\chi(\phi(z))^{\,m}=\lambda^{m}\chi(z)^{\,m}$,
so $\chi(z)^{\,m}$ is an eigenfunction for $\lambda^{m}$, and
$\overline{\chi(z) }^{\,m}$ for $\overline{\lambda}^{\,m}$.

Let $0<\theta<1$. Since $\chi$ is analytic on $\mathbb{D}$, for
$\theta<r<1$ Cauchy Theorem gives the estimate for its Fourier coefficients 
\[|\chi_k|\le M(r)r^{-k}\] with
$M(r)=\sup_{|z|=r}|\chi(z)|$. So,
\[
\|\chi\|^{2}_{\theta^{-1}}=\sum_{k\ge1}|\chi_k|^{2}\theta^{2k}
\le M(r)^{2}\sum_{k\ge1}\left(\theta/r\right)^{2k}<\infty ,
\]
and the same holds for $\chi(z)^{\,m}$ and $\overline{\chi(z) }^{\,m}$. 

Since $\phi$ preserves $\mu$, $\mathcal{L}$ is an isometry of $L_2(\mathbb{S}^1,\mu)$. If $\chi$ were in $L_2$, then $\chi\circ\phi=\lambda\chi$
would give $\|\chi\|_{L_2}=|\lambda|\,\|\chi\|_{L_2}$, and $|\lambda|<1$ would
force $\chi=0$. The same holds for $\chi(z)^{\,m}$ and $\overline{\chi(z) }^{\,m}$.
\end{proof}

The last statement is another reason for working with hyperfunctions. The
eigenvalues $\lambda^{m}$ are carried by eigenfunctions lying in every space
$\mathcal{H}_{\theta^{-1}}$ and not in $L_2$. 

Operators with zero spectral radius are called \textit{quasi-nilpotent}. By
Theorem~\ref{thm:spectrum} this happens exactly when $\phi'(0)=0$, and then
$\sigma(\mathcal{L};\mathcal{H}^{0}_{\theta^{-1}})=\{0\}$. These
operators are \textit{non-self-adjoint}, so a trivial spectrum does not make them
trivial. Quasi-nilpotency implies \textit{super
exponential mixing} (see Corollary~\ref{super}). In general spectral theory, quasi-nilpotency does not imply mixing at a double exponential rate:  the classical Volterra operator $V$ on $L_2[0,1]$ is quasi-nilpotent, with
$\langle V^{n}1,1\rangle=\frac{1}{(n+1)!}$, a decay of order $e^{-n\log n}$ that
is faster than any exponential but not double exponential.

Next, we discuss the case of several generators.

\begin{theorem}[Joint spectral radius]\label{prop:jsr}
Let $\phi_1,\dots,\phi_N\in\mathcal{B}^0$ and 
$\lambda_j:=\phi_j'(0)$, and assume that $|\lambda_j|<1$ for all $1\le j\le N$.
Then the
joint spectral radius of $\{\mathcal{L}_1,\dots,\mathcal{L}_N\}$,
\[
\rho:=\lim_{n\to\infty}\ \max_{|a|=n}\ \|\mathcal{L}^{a}\|^{1/n}_{\mathcal{H}^{0}_{\theta^{-1}}\to\mathcal{H}^{0}_{\theta^{-1}}},
\]
equals
\[
\rho=\max_{1\le j\le N}|\lambda_j| .
\]
If moreover $\lambda_j=0$ for all $j$, so that $\rho=0$, then every
non-commutative polynomial in $\mathcal{L}_1,\dots,\mathcal{L}_N$ with zero
constant term is quasi-nilpotent; that is, $\mathcal{L}_1,\dots,\mathcal{L}_N$
generate a \textit{quasi-nilpotent non-unital algebra}.
\end{theorem}

\begin{proof}
For any word $a$, $\tau(a)$ is again a finite Blaschke product fixing the origin,
with \[\tau(a)'(0)=\prod_i\phi_{a_i}'(0)\] by the chain rule, so
Theorem~\ref{thm:spectrum} applied to $\tau(a)$ gives
\[r(\mathcal{L}^{a})=\bigl|\prod_i\phi_{a_i}'(0)\bigr|.\] Hence
\[
\max_{|a|=n}\ r(\mathcal{L}^{a})=\Bigl(\max_{1\le j\le N}|\lambda_j|\Bigr)^{n}
\qquad\text{for every } n\ge1 .
\]
Since the $\mathcal{L}_j$ are compact (the proof of Theorem~\ref{norm bound}
bounds their Hilbert--Schmidt norms), the Berger--Wang formula \cite{ST00}
applies and gives
\[
\rho=\limsup_{n\to\infty}\ \max_{|a|=n}\ r(\mathcal{L}^{a})^{1/n}
=\max_{1\le j\le N}|\lambda_j| .
\]

Suppose now $\lambda_j=0$ for all $j$. Then $\rho=0$, which also follows directly
from Theorem~\ref{norm bound}, whose bound is uniform over the words of a given
length, so
\[\max_{|a|=n}\|\mathcal{L}^{a}\|^{1/n}_{\mathcal{H}^{0}_{\theta^{-1}}\to\mathcal{H}^{0}_{\theta^{-1}}}\le C^{1/n}e^{-\beta e^{\gamma n}/n}\to0.\]
 If $P$ is a non-commutative polynomial in $\mathcal{L}_1,\dots,\mathcal{L}_N$
with zero constant term \[P=\sum_{i=1}^{m}c_i\mathcal{L}^{a_i} ,\qquad |a_i|\ge1,\] 
then $P^{n}$ is a sum of words of length at least $n$, so
\[
\|P^{n}\|_{\mathcal{H}^{0}_{\theta^{-1}}\to\mathcal{H}^{0}_{\theta^{-1}}}\le\left(\sum_{i=1}^{m}|c_i|\right)^{n}\sup_{|a|\ge n}\|\mathcal{L}^{a}\|_{\mathcal{H}^{0}_{\theta^{-1}}\to\mathcal{H}^{0}_{\theta^{-1}}}
\]
and \[r(P)=\lim_{n\to\infty}\|P^{n}\|^{1/n}_{\mathcal{H}^{0}_{\theta^{-1}}\to\mathcal{H}^{0}_{\theta^{-1}}}\le \sum_{i=1}^{m}|c_i|\lim_{n\to\infty} C^{1/n}\sup_{|a|\ge n} e^{-\frac{\beta}{n} e^{\gamma |a|}} =0.\]
\end{proof}

Note that, a priori, a composition of non-commuting quasi-nilpotent operators
does not have to be quasi-nilpotent: the matrices
$A=\begin{psmallmatrix}0&1\\0&0\end{psmallmatrix}$ and
$B=\begin{psmallmatrix}0&0\\1&0\end{psmallmatrix}$ are nilpotent, while already
the word $AB=\begin{psmallmatrix}1&0\\0&0\end{psmallmatrix}$ is a nonzero
idempotent. The last statement of Theorem~\ref{prop:jsr} is therefore
stronger than quasi-nilpotency of the individual generators, and does not follow
from it. In this way the actions considered here provide a naturally occurring
example of a quasi-nilpotent non-unital algebra of operators that need not
commute. 

It is natural to ask what happens when the generators commute. Then $\tau$
factors through $\mathbb{Z}_+^{N}$, and, apart from one degenerate case, the
action is forced to be affine.

\begin{prop}[Commuting generators]\label{prop:commuting}
Let $\phi_1,\dots,\phi_N\in\mathcal{B}^0$, $N\ge2$, be of degree at least two, and suppose that they commute pairwise. Then either every two generators have a common
iterate, $\phi_i^{\,p}=\phi_j^{\,q}$ for some $p,q\in\mathbb{N}$, or every generator
is affine,
\[
\phi_j(z)=e^{i\psi_j}z^{q_j},\qquad \psi_j\in[0,2\pi),\quad q_j\ge2 .
\]
\end{prop}

\begin{proof}
Fix $i\neq j$. Since $\phi_i$ and $\phi_j$ commute and $\phi_i(0)=0$ for all $1\le i\le N$, Chalendar and Mortini
\cite{CM01} prove that either
$\phi_i^{\,p}=\phi_j^{\,q}$ for some $p,q\in\mathbb{N}$ or $\phi_i$ and $\phi_j$ become affine after a M\"obius conjugation. Since $\phi_i(0)=0$, that conjugation is a rotation, and $\phi_i$ is itself
affine, $\phi_i(z)=e^{i\psi_i}z^{q_i}$. Thus each pair either has a common iterate or consists of two affine
maps.

Suppose now that some generator is affine, say $\phi_1$, and let $j\neq1$. If the
pair $(1,j)$ has a common iterate, $\phi_1^{\,p}=\phi_j^{\,q}$, then $\phi_j^{\,q}$
is affine and so is $\phi_j$.
In either case $\phi_j$ is affine. So one affine generator forces all of them, and
if none is affine, every pair has a common iterate.
\end{proof}

This is the one-dimensional form of a general
phenomenon. In \cite{SY17} Spatzier and Yang prove that any genuinely higher rank action of
an abelian semigroup by expanding maps of a compact manifold is $C^{\infty}$-conjugate to an affine action on an infra-nilmanifold.  

In the affine case, we have $\phi_j'(0)=0$ for every $1\le j\le N$, so by Theorem~\ref{prop:jsr} the
joint spectral radius is zero. In the remaining case, when every two generators
share an iterate, differentiating $\phi_i^{\,p}=\phi_j^{\,q}$ at the origin gives
$\lambda_i^{\,p}=\lambda_j^{\,q}$ for the multipliers $\lambda_j=\phi_j'(0)$. So
either they all vanish, and the joint spectral radius is zero again, or none of them
does, and the joint spectrum of the commuting tuple on $\mathcal{H}^{0}_{\theta^{-1}}$
is
\[
\bigl\{(\lambda_1^{m},\dots,\lambda_N^{m})\bigr\}_{m\ge1}\cup
\bigl\{(\overline{\lambda_1}^{\,m},\dots,\overline{\lambda_N}^{\,m})\bigr\}_{m\ge1}
\cup\{0\}\ \subseteq\ \mathbb{C}^{N},
\]
a discrete subset of the polydisk accumulating only at the origin.

%% file: harmonic_fixed.tex
\section{Blaschke Products with a Fixed Point in the Disk}\label{Other invariant measures}

Throughout, $\phi_1,\dots,\phi_N$ are finite Blaschke products with a common
fixed point $w\in\mathbb{D}$, and $\tau$ is the free semigroup action they generate. Put
\[
h(z):=\frac{z+w}{1+\overline{w}z},\qquad
\widetilde{\phi}_j:=h^{-1}\circ\phi_j\circ h,\qquad
\widetilde{\tau}(a):=h^{-1}\circ\tau(a)\circ h ,
\]
so that every $\widetilde\phi_j$ lies in $\mathcal{B}^0$ and
$\widetilde\tau(a)=\widetilde\phi_{a_{|a|}}\circ\dots\circ\widetilde\phi_{a_1}$:
one M\"obius map normalizes the whole action at once. Let $\nu_w$ be the harmonic
measure with pole at $w$; by Corollary~\ref{cor:harmonic measure preserving} it
is invariant under each $\phi_j$, hence under every $\tau(a)$, and $\nu_w=h_*\mu$.

\begin{proof}[Proof of Theorem~\ref{classification}, general case]
Write $\mathcal{C}^{\nu_w}_{f,g}$ for the correlation function computed with
respect to $\nu_w$. Since $\nu_w=h_*\mu$ and
$\tau(a)\circ h=h\circ\widetilde\tau(a)$,
\begin{equation}\label{eq:correlation transfer}
\mathcal{C}^{\nu_w}_{f,g}(a)=\mathcal{C}_{f\circ h,\;g\circ h}(a),
\end{equation}
the right-hand side being computed with respect to $\mu$ for the action
$\widetilde\tau$. Moreover a conformal conjugation preserves multipliers and
local degrees, so $\widetilde\phi_j'(0)=\phi_j'(w)$, and the local degree $p_j$
of $\phi_j$ at $w$ equals that of $\widetilde\phi_j$ at the origin; we put
$p:=\min_j p_j$.

If $|\phi_j'(w)|=1$ for some $j$, then $\phi_j$ is an elliptic automorphism and
$\widetilde\phi_j$ is a rotation, so $\widetilde\tau$, and hence $\tau$, is not
mixing; mixing is a measure-theoretic property, and $\nu_w=h_*\mu$. If
$|\phi_j'(w)|<1$ for all $j$, then $h$ is an analytic diffeomorphism of $\mathbb{S}^1$,
so $f\mapsto f\circ h$ preserves $C^{\alpha}$ with equivalent norms, and
exponential mixing of $\widetilde\tau$ on $C^{\alpha}$ gives that of $\tau$.

There remains the case $\phi_j'(w)=0$ for all $j$, that is $p\ge2$. The space $\mathcal{H}_\theta$
is not preserved by precomposition with $h$, since a function of
$\mathcal{H}_\theta$ is analytic on the annulus $\{\theta<|z|<1/\theta\}$,
whereas $f\circ h$ is analytic only on the preimage of that annulus under $h$,
containing a smaller annulus. We must therefore pass to a larger parameter $\theta'$.
Lemma~\ref{lm:conjugation on analytic spaces} below provides one, for every
$\theta$: it gives $\theta'$ and a constant $K=K(\theta,\theta',w)$, independent
of $f$, such that
\[
f\in\mathcal{H}_\theta\ \Longrightarrow\ f\circ h\in\mathcal{H}_{\theta'}
\quad\text{and}\quad \|f\circ h\|_{\theta'}\le K\|f\|_\theta.
\]
 In particular $f\circ h$ and $g\circ h$
belong to $\mathcal{H}_{\theta'}$, so the normalized case, proved in
Sections~\ref{proof} and \ref{one-more-proof}, applies to $\widetilde\tau$ at the
parameter $\theta'$ and gives, together with
\eqref{constants for double exponential mixing},
\begin{equation}\label{eq:harmonic bound}
\mathcal{C}^{\nu_w}_{f,g}(a)=\mathcal{C}_{f\circ h,\,g\circ h}(a)
\ \le\ C_{\theta'}\,\theta'^{\,p^{|a|}}\|f\circ h\|_{\theta'}\|g\circ h\|_{\theta'}
\ \le\ C_{\theta'}K^{2}\,\theta'^{\,p^{|a|}}\,\|f\|_\theta\|g\|_\theta
\end{equation}
for all $a\in\mathbb{F}_N^+$, that is, $\tau$ mixes at a double exponential rate
with respect to $\nu_w$, with $\gamma=\log p$ and $\beta=\log\frac1{\theta'}$.

Finally, mixing at a double exponential rate forces
$\phi_j'(w)=0$ for all $j$: this is
Lemma~\ref{lower bounds on exponential mixing} transferred to the non-normalized
case by the choice $f=g=h^{-1}$, for which $(f\circ h)(z)=(g\circ h)(z)=z$, so that
\eqref{eq:correlation transfer} turns the lower bound for $\widetilde\tau$ into
the same lower bound for $\tau$. These observables lie in
$\mathcal{H}_{\theta''}$ for every $|w|<\theta''<1$, and by
Lemma~\ref{lm:theta independence} one such parameter suffices.
\end{proof}

It remains to prove the norm comparison used above.

\begin{lemma}\label{lm:conjugation on analytic spaces}
Let $0<\theta<1$, $w\in\mathbb{D}$, and let $\theta'$ satisfy
\begin{equation}\label{eq:conjugation threshold}
\frac{\theta+|w|}{1+\theta|w|}\ <\ \theta'\ <\ 1 .
\end{equation}
Then there is $K=K(\theta,\theta',w)$ such that for any $f\in\mathcal{H}_\theta$
\begin{equation}\label{eq:norm comparison}
\|f\circ h\|_{\theta'}\ \le\ K\,\|f\|_{\theta} .
\end{equation}
\end{lemma}

\begin{proof}
Put
\begin{equation}\label{eq:sigma}
\sigma:=\frac{1-|w|\theta'}{\theta'-|w|} .
\end{equation}
The interval in \eqref{eq:conjugation threshold} is nonempty for every
$\theta<1$ and every $w\in\mathbb{D}$, since $\theta+|w|<1+\theta|w|$ is equivalent to
$\theta(1-|w|)<1-|w|$. Note that \eqref{eq:conjugation threshold} implies $\theta'>|w|$, so $\sigma>0$ and $\theta\sigma<1$.

First, let us show that
\[
\min_{|z|=\theta'}|h(z)|=\frac{1}{\sigma} \quad\text{and}\quad \max_{|z|=1/\theta'}|h|=\sigma.
\]
After a
rotation we may assume
$w=|w|\ge0$, so that $h$ preserves the real axis. A M\"obius map takes circles to
circles, so $h$ takes $\{|z|=r\}$ to a circle symmetric about the real axis,
whose centre therefore lies on that axis; the extrema of $|h|$ on $\{|z|=r\}$ are
thus attained at $z=\pm r$. In particular, taking $r=\theta'>w$,
\[
\min_{|z|=\theta'}|h(z)|=|h(-\theta')|=\frac{\theta'-w}{1-w\theta'}=\frac{1}{\sigma},
\]
and the reflection identity $h(1/\overline{z})=1/\overline{h(z)}$, valid for
automorphisms of $\mathbb{D}$, gives $\max_{|z|=1/\theta'}|h|=\sigma$.

Next, we estimate norms of powers of $h$. Let $k\ge0$. The only pole of $h$ is at
$-1/\overline{w}$ and $1/\theta'<1/|w|$, so $(h(z))^{k}$ is analytic on
$|z|\le1/\theta'$; hence $j$-th Fourier coefficient of $(h(z))^{k}$ is zero for $j<0$, and Parseval on the
circle $|z|=1/\theta'$ gives
\[
\|(h(z))^{k}\|_{\theta'}^{2}=\sum_{j\ge0}\bigl|(h(z))^{k}_{j}\bigr|^{2}\Bigl(\frac{1}{\theta'}\Bigr)^{2j}
=\frac{1}{2\pi}\int_{0}^{2\pi}\bigl|h(\theta'^{-1}e^{i\vartheta})\bigr|^{2k}d\vartheta
\le\sigma^{2k} .
\]
For $k<0$ the function $(h(z))^{k}=(1/h(z))^{|k|}$ is analytic on $|z|\ge\theta'$,
including at infinity, since the only zero of $h$ is at $-w$; the same argument
on the circle $|z|=\theta'$ gives $\|(h(z))^{k}\|_{\theta'}\le\sigma^{|k|}$. Thus for all $k\in\mathbb{Z}$
\begin{equation}\label{eq:powers of h}
\|(h(z))^{k}\|_{\theta'}\le\sigma^{|k|}.
\end{equation}

Finally, we estimate the norms. Writing $f(z)=\sum_{k\in\mathbb{Z}}f_kz^{k}$, so that \[(f\circ h)(z)=\sum_{k\in\mathbb{Z}}f_k(h(z))^{k},\]
the triangle inequality, \eqref{eq:powers of h} and the Cauchy--Schwarz
inequality give
\[
\|f\circ h\|_{\theta'}\le\sum_{k\in\mathbb{Z}}|f_k|\,\|(h(z))^{k}\|_{\theta'}
\le\sum_{k\in\mathbb{Z}}\bigl(|f_k|\theta^{-|k|}\bigr)(\theta\sigma)^{|k|}
\le\sqrt{\sum_{k\in\mathbb{Z}}(\theta\sigma)^{2|k|}}\cdot\|f\|_\theta ,
\]
the last series converging precisely because $\theta\sigma<1$. %The same bound shows that $\sum_kf_kh^{k}$ converges absolutely in $\mathcal{H}_{\theta'}$, which justifies the first inequality.
Thus we get \eqref{eq:norm comparison} with 
\[
K=\sqrt{\sum_{k\in\mathbb{Z}}(\theta\sigma)^{2|k|}}
=\sqrt{\frac{1+(\theta\sigma)^{2}}{1-(\theta\sigma)^{2}}}.
\]
\end{proof}

The same conjugation transfers the remaining results of the normalized case.
Applying it to a generator of minimal local degree,
Theorem~\ref{thm:sharpness} shows that the exponent $\gamma=\log p$ in
\eqref{eq:harmonic bound} is optimal; and the rigidity statement of Corollary~\ref{cor:rigidity fixed degree} states that a single generator of degree $q$
whose exponent attains $\log q$ satisfies
$\phi=h\circ\bigl(e^{i\psi}z^{q}\bigr)\circ h^{-1}$, that is, it is M\"obius
conjugate to an affine model. The results of Section~\ref{spec-qnilp-res} transfer unchanged, since conjugation
by $h$ preserves the multiplier: the spectrum of $\mathcal{L}$ is
$\{0,1\}\cup\{\phi'(w)^{m},\overline{\phi'(w)}^{\,m}:m\ge1\}$, with
eigenfunctions the powers of the K\"onigs function of $\phi$ at $w$.

%% file: Invariance_of_Lebesgue.tex
\section{Invariance of the Harmonic Measure}\label{app:Lebesgue measure preserving}

\begin{lemma}\label{lm:Leb meas preserv}
    Let $\phi\in\mathcal{B}$ be a finite Blaschke product. Then $\phi$ preserves the Lebesgue measure $\mu$ on $\mathbb{S}^1$ if and only if
    \[
    \phi(0)=0.
    \]
\end{lemma}

\begin{proof}
A finite Borel measure on $\mathbb{S}^1$ is determined by its Fourier coefficients (see \cite{Kat04}). Applying this to the measures $\phi_*\mu$ and $\mu$, and using \[\int_{\mathbb{S}^1}z^k\,d\mu(z)=\delta_k,\] we see that the measure-preserving property can be checked on the characters $z^k$. The map $\phi$ preserves $\mu$ if and only if
\begin{equation}\label{eq:Fourier criterion}
\int_{\mathbb{S}^1} (\phi (z))^k\, d\mu(z) = \delta_k:=
\begin{cases}
1, & k=0,\\
0, & k\neq 0,
\end{cases}
\qquad \text{for every }\quad k\in \mathbb{Z}.
\end{equation}

It remains to compute the left-hand side. For integer $k\ge 0$, we apply the Cauchy Theorem to $(\phi(z))^k$, a function that is analytic in the unit disk, and get
\begin{align*}
\int_{\mathbb{S}^1}(\phi(z))^k\,d\mu(z)&=\frac{1}{2\pi i}\int_{\mathbb{S}^1}(\phi(z))^k\,\frac{dz}{z}\\&=\left(\phi(0)\right)^k.
\end{align*}
For integer $k<0$, we use $|\phi(z)|=1$ for $|z|=1$ to get
\begin{align*}
\int_{\mathbb{S}^1}(\phi(z))^k\,d\mu(z)&=\int_{\mathbb{S}^1}(\phi(z))^{-|k|}\,d\mu(z)
\\&=\int_{\mathbb{S}^1}(\overline{\phi(z)})^{|k|}\,d\mu(z)\\&=\overline{\int_{\mathbb{S}^1}(\phi(z))^{|k|}\,d\mu(z)}
\\&=\left(\overline{\phi(0)}\right)^{|k|},
\end{align*}
where in the last line we used that $\mu$ is a positive real measure.

From here we see that \eqref{eq:Fourier criterion} holds if and only if $\left(\phi(0)\right)^k=0$ for every $k\ge1$, that is, if and only if $\phi(0)=0$. Thus, $\phi(0)=0$ is equivalent to $\phi$-invariance of $\mu$.
\end{proof}

The general case follows by a conformal change of variable.

\begin{corollary}\label{cor:harmonic measure preserving}
Let $w\in\mathbb{D}$ and let
\[
d\nu_w(z):=\frac{1-|w|^{2}}{|z-w|^{2}}\,d\mu(z)
\]
be the harmonic measure on $\mathbb{S}^1$ with pole at $w$. Let $\phi\in\mathcal{B}$
be a finite Blaschke product. Then $\phi$ preserves $\nu_w$ if and only if
\[
\phi(w)=w .
\]
\end{corollary}
\begin{proof}
Let
\[
h(z):=\frac{z+w}{1+\overline{w}z},
\]
an automorphism of the unit disk with $h(0)=w$. 

Since $d\mu(\zeta)=\frac{1}{2\pi i}\frac{d\zeta}{\zeta}$, substituting
$\zeta=h^{-1}(z)$ and using
\[
h^{-1}(z)=\frac{z-w}{1-\overline{w}z},
\qquad
\left(h^{-1}\right)'(z)=\frac{1-|w|^{2}}{\left(1-\overline{w}z\right)^{2}},
\]
we get
\[
d(h_*\mu)(z)=d\mu\bigl(h^{-1}(z)\bigr)
=\frac{1}{2\pi i}\,\frac{\left(h^{-1}\right)'(z)}{h^{-1}(z)}\,dz
=\frac{1}{2\pi i}\,\frac{(1-|w|^{2})\,dz}{(z-w)(1-\overline{w}z)} .
\]
For $|z|=1$ we have 
\[(z-w)(1-\overline{w}z)=z\,|z-w|^{2},\] and therefore
\[
d(h_*\mu)(z)=\frac{1-|w|^{2}}{|z-w|^{2}}\cdot\frac{1}{2\pi i}\frac{dz}{z}=d\nu_w(z).
\]

Let $\widetilde\phi:=h^{-1}\circ\phi\circ h$, again a finite Blaschke
product, so
\[
\phi_*\nu_w=(\phi\circ h)_*\mu=(h\circ\widetilde\phi)_*\mu=h_*\bigl(\widetilde\phi_*\mu\bigr),
\]
and $h_*$ is injective, so $\phi_*\nu_w=\nu_w=h_*\mu$ if and only if
$\widetilde\phi_*\mu=\mu$. By Lemma~\ref{lm:Leb meas preserv} this
holds if and only if $\widetilde\phi(0)=0$, that is, $\phi(w)=w$.
\end{proof}

\begin{comment}
\begin{rem}\label{rem:inner functions}
Note that in the proof of Lemma~\ref{cor:volume preserving Blaschke} we only used that $T$ is analytic in the unit disk and $|T|=1$ on the unit circle, that is, that $T$ is an \textit{inner function}. For $k\ge0$ one replaces the Cauchy Theorem by the corresponding $H^1$ statement: the Fourier coefficients of the boundary function of $T^k\in H^\infty$ coincide with its Taylor coefficients, and the $0$-th one is $\left(T(0)\right)^k$. Therefore, an inner function preserves $\mu$ if and only if it fixes the origin.
\end{rem}
\end{comment}

%% file: some_theta_any_theta.tex
\section{Mixing on Analytic Functions}\label{some theta any theta}

Let $\tau$ be the action of the free semigroup $\mathbb{F}_N^+$ on $\mathbb{S}^1$
generated by $\phi_1,\dots,\phi_N\in\mathcal{B}$ with a common fixed point
$w\in\mathbb{D}$, let $\nu_w$ be the harmonic measure of
\eqref{eq:harmonic measure}, and let the correlation function
$\mathcal{C}_{f,g}$ be computed with respect to $\nu_w$. 

\begin{lemma}\label{lm:theta independence}
Suppose that for \textit{some}
$0<\theta_0<1$ there exist $C,\beta,\gamma>0$ such that for all $f,g\in\mathcal{H}_{\theta_0}$ and $a\in\mathbb{F}_N^+$
\[
\mathcal{C}_{f,g}(a)\le C e^{-\beta e^{\gamma|a|}}\|f\|_{\theta_0}\|g\|_{\theta_0}.
\]
Then for \textit{every} $0<\theta<1$ there exist $C_\theta,\beta_\theta>0$, with
the same $\gamma$, such that for all $f,g\in\mathcal{H}_{\theta}$ and $a\in\mathbb{F}_N^+$ 
\[
\mathcal{C}_{f,g}(a)\le C_\theta e^{-\beta_\theta e^{\gamma|a|}}\|f\|_{\theta}\|g\|_{\theta}.
\]
\end{lemma}

\begin{proof}

We first give the argument in the normalized case $w=0$.

Let $f,g\in\mathcal{H}_\theta$ and assume without loss of generality that $f$ has
an average of zero.

Suppose first that $\theta\le\theta_0$. Then $\theta_0^{-2|k|}\le\theta^{-2|k|}$
for every $k$, so $\mathcal{H}_\theta\subseteq\mathcal{H}_{\theta_0}$ and
$\|f\|_{\theta_0}\le\|f\|_\theta$. The claim follows with
$C_\theta=C$ and $\beta_\theta=\beta$.

Suppose now that $\theta>\theta_0$. For $M\in\mathbb{N}$ denote by
$f_M:=\sum_{|k|\le M}f_kz^k$ the truncation of the Fourier series of $f$. %notethat $f_M$ again has an average of zero. 
We have
\begin{equation}\label{eq:truncation cost}
\|f_M\|_{\theta_0}^2=\sum_{|k|\le M}|f_k|^2\theta^{-2|k|}\Bigl(\tfrac{\theta}{\theta_0}\Bigr)^{2|k|}
\le \Bigl(\tfrac{\theta}{\theta_0}\Bigr)^{2M}\|f\|_\theta^2 ,
\end{equation}
since $\theta/\theta_0>1$, and
\begin{equation}\label{eq:truncation error}
\|f-f_M\|_{L_2}^2=\sum_{|k|>M}|f_k|^2\theta^{-2|k|}\theta^{2|k|}\le \theta^{2M}\|f\|_\theta^2 .
\end{equation}
Note also that $\|f\|_{L_2}\le\|f\|_\theta$, because $\theta^{-2|k|}\ge1$.

We split
\begin{equation}\label{eq:split theta}
\langle\mathcal{L}^a f,g\rangle_{L_2}=\langle\mathcal{L}^a f,g-g_M\rangle_{L_2}+\langle\mathcal{L}^a(f-f_M),g_M\rangle_{L_2}+\langle\mathcal{L}^a f_M,g_M\rangle_{L_2}.
\end{equation}
Since $\mathcal{L}^a$ is an isometry of $L_2(\mathbb{S}^1,\mu)$, the first two
terms are estimated by \eqref{eq:truncation error},
\[
|\langle\mathcal{L}^a f,g-g_M\rangle_{L_2}|\le\|f\|_{L_2}\|g-g_M\|_{L_2}\le\theta^{M}\|f\|_\theta\|g\|_\theta ,
\]
\[
|\langle\mathcal{L}^a(f-f_M),g_M\rangle_{L_2}|\le\|f-f_M\|_{L_2}\|g\|_{L_2}\le\theta^{M}\|f\|_\theta\|g\|_\theta .
\]
For the third term we use the hypothesis at $\theta_0$, which applies because
$f_M,g_M$ are trigonometric polynomials and $f_M$ has average zero, together
with \eqref{eq:truncation cost},
\[
|\langle\mathcal{L}^a f_M,g_M\rangle_{L_2}|\le Ce^{-\beta e^{\gamma|a|}}\|f_M\|_{\theta_0}\|g_M\|_{\theta_0}
\le Ce^{-\beta e^{\gamma|a|}}\Bigl(\tfrac{\theta}{\theta_0}\Bigr)^{2M}\|f\|_\theta\|g\|_\theta .
\]
Combining these with \eqref{eq:split theta}, we get for every $M$
\begin{equation}\label{eq:two terms theta}
\mathcal{C}_{f,g}(a)\le\Bigl(2\theta^{M}+Ce^{-\beta e^{\gamma|a|}}\bigl(\tfrac{\theta}{\theta_0}\bigr)^{2M}\Bigr)\|f\|_\theta\|g\|_\theta .
\end{equation}
We balance the two terms by choosing
\[
M(|a|):=\left\lceil K e^{\gamma|a|}\right\rceil,\qquad K=\frac{\beta}{4\log(\theta/\theta_0)}>0,
\]
so that $Ke^{\gamma|a|}\le M(|a|)\le Ke^{\gamma|a|}+1$. Then
\[
\Bigl(\tfrac{\theta}{\theta_0}\Bigr)^{2M(|a|)}\le\Bigl(\tfrac{\theta}{\theta_0}\Bigr)^{2}e^{2K\log(\theta/\theta_0)e^{\gamma|a|}}
=\Bigl(\tfrac{\theta}{\theta_0}\Bigr)^{2}e^{\frac{\beta}{2}e^{\gamma|a|}},
\]
so the second term in \eqref{eq:two terms theta} is at most
$C\bigl(\theta/\theta_0\bigr)^{2}e^{-\frac{\beta}{2}e^{\gamma|a|}}$, while the
first one is at most
\[
2\,\theta^{Ke^{\gamma|a|}}=2\,e^{-K\log\frac1\theta\,e^{\gamma|a|}} .
\]
Both decay double exponentially in $|a|$ with the same exponent $\gamma$, so
\[
\mathcal{C}_{f,g}(a)\le C_\theta e^{-\beta_\theta e^{\gamma|a|}}\|f\|_\theta\|g\|_\theta ,
\]
where
\[
C_\theta=2+C\Bigl(\tfrac{\theta}{\theta_0}\Bigr)^{2},
\qquad
\beta_\theta=\min\Bigl\{\tfrac{\beta}{2},\ \tfrac{\beta\log(1/\theta)}{4\log(\theta/\theta_0)}\Bigr\}>0 . 
\]

Now we discuss the changes needed for the general case $w\in\mathbb{D}$. The argument above used only that $\mathcal{L}^a$ is an isometry of the $L_2$ space of the
invariant measure, and that $\|\cdot\|_{L_2}\le\|\cdot\|_\theta$; the latter now
holds up to the factor
\[
D:=\sup_{\mathbb{S}^1}\left(\frac{d\nu_w}{d\mu}\right)^{1/2}
=\left(\frac{1+|w|}{1-|w|}\right)^{1/2} .
\] 
Two changes are needed, and both
affect only $C_\theta$: the $L_2(\nu_w)$ and $L_2(\mu)$ norms differ by the factor
$D$, and a truncation $f_M$ of a function with vanishing
$\nu_w$-average need no longer have vanishing $\nu_w$-average. The last term of \eqref{eq:split theta} is estimated by
\[\mathcal{C}_{f_M,g_M}(a)+\left|\int f_M\,d\nu_w\right|\left|\int g_M\,d\nu_w\right|,\]
Since $\int f\,d\nu_w=0$ by assumption, and since $\nu_w$ is a probability measure,
\[
\left|\int f_M\,d\nu_w\right|=\left|\int(f_M-f)\,d\nu_w\right|
\le\|f_M-f\|_{L_2(\nu_w)}\le D\|f_M-f\|_{L_2(\mu)}\le D\,\theta^{M}\|f\|_\theta
\]
by \eqref{eq:truncation error}, while the second factor is bounded
\[\left|\int g_M\,d\nu_w\right|\le\|g_M\|_{L_2(\nu_w)}\le D\|g\|_\theta.\] Thus
\[
\left|\int f_M\,d\nu_w\right|\left|\int g_M\,d\nu_w\right|\le D^{2}\theta^{M}\|f\|_\theta\|g\|_\theta ,
\]
which is of the same order as the first two terms of \eqref{eq:split theta}.
\end{proof}

The exponent $\gamma$ is the same for all $\theta$, while $\beta$ and $C$ are
not. Thus $\gamma$ is an invariant of the action alone, in accordance with
$\gamma=\log p$ of \eqref{constants for double exponential mixing}, where $p$ is
the minimal local degree of the generators at the origin.

%% file: Mixing_Hierarchy.tex
\section{Mixing Hierarchy}\label{hierarchy}

Let $\tau$ be the action of the free semigroup $\mathbb{F}_N^+$ on $\mathbb{S}^1$ generated by $\phi_1,\dots,\phi_N\in\mathcal{B}^0$, as in Section~\ref{ch1:results}. Since each $\phi_j$ preserves the Lebesgue measure $\mu$, every operator $\mathcal{L}^a$, $a\in\mathbb{F}_N^+$, is an isometry of $L_2(\mathbb{S}^1,\mu)$.

\begin{lemma}[Mixing Hierarchy]\label{Lm: mixing hierarchy}
If $\tau$ is mixing at a double exponential rate on $\mathcal{H}_\theta$ for some $0<\theta<1$, then $\tau$ is exponentially mixing on $C^{\alpha}$ for every $0<\alpha\le1$.
\end{lemma}
Note that the converse does not hold, as shown by examples in Section~\ref{examples}.
In the proof, we use an approximation argument similar to Lemma~2.3 in \cite{GS14}. 

\begin{proof} Assume first that $0<\alpha<1$. Let $f,g\in C^\alpha$, and assume without loss of generality that $f$ has an average of zero.
Denote by $\sigma_M f$ the $M$-th Fej\'er mean of the Fourier series of $f$. There is $C'=C'(\alpha)>0$ such that for any $f\in C^\alpha$ and $f_M=\sigma_M f$,
\[
\| f_M - f \|_{\infty} \;\le\; C'\, M^{-\alpha}\|f\|_{C^\alpha}
\]
(see \cite[Ch.~I, \S3, Exercise~2]{Kat04}). Note that $\sigma_M$ does not change the zeroth Fourier coefficient, so $f_M$ also has average zero.
We have
\begin{equation}\label{all together}
\langle\mathcal{L}^a f,g\rangle_{L_2}=\langle\mathcal{L}^a f,g-g_M\rangle_{L_2}+\langle\mathcal{L}^a(f-f_M),g_M\rangle_{L_2}+\langle\mathcal{L}^a f_M,g_M\rangle_{L_2}.
\end{equation}
Now, being finite linear combinations of basis vectors, $f_M$ and $g_M$ belong to any $\mathcal{H}_\theta$,
so there are $C>0$, $\beta>0$, and $\gamma>0$ such that
\begin{align}\label{first one}
|\langle\mathcal{L}^a f_M,g_M\rangle_{L_2}|&\le Ce^{-\beta e^{\gamma |a|}}\|f_M\|_\theta\|g_M\|_\theta
\nonumber \\ &\le Ce^{-\beta e^{\gamma
|a|}}\theta^{-2M}\|f_M\|_{L_2}\|g_M\|_{L_2}
\nonumber\\
&\le Ce^{-\beta e^{\gamma |a|}}\theta^{-2M}\|f\|_{L_2}\|g\|_{L_2}
\nonumber\\
&\le Ce^{-\beta e^{\gamma |a|}}\theta^{-2M}\|f\|_{C^\alpha}\|g\|_{C^\alpha}.
\end{align}
Next, using that $\mathcal{L}^a$ is an $L_2$-isometry,
\begin{equation}\label{second one}
|\langle\mathcal{L}^a f,g-g_M\rangle_{L_2}|\le \|g-g_M\|_{\infty}\|\mathcal{L}^af\|_{L_2}\le C' M^{-\alpha}\|f\|_{C^\alpha}\|g\|_{C^\alpha},
\end{equation}
and similarly,
\begin{equation}\label{third one}
|\langle\mathcal{L}^a(f-f_M),g_M\rangle_{L_2}|\le C'  M^{-\alpha}\|f\|_{C^\alpha}\|g\|_{C^\alpha}.
\end{equation}
Combining the estimates \eqref{all together}--\eqref{third one}, we get
\begin{equation}\label{final}
    |\langle\mathcal{L}^a f, g\rangle_{L_2}| \le \left(2C'M^{-\alpha}+Ce^{-\beta e^{\gamma |a|}}\theta^{-2M}\right)\|f\|_{C^\alpha}\|g\|_{C^\alpha}.
\end{equation}
There are two competing terms that we need to balance by choosing $M$ appropriately as a function of $|a|$
\[
M(|a|):=\left\lceil K e^{\gamma |a|} \right\rceil,\quad  K=\frac{\beta}{4\log (1/\theta)}>0,
\]
so that
\begin{equation}\label{M bounds}
K e^{\gamma |a|}\;\le\; M(|a|)\;\le\; K e^{\gamma |a|}+1 .
\end{equation}
Using the upper bound in \eqref{M bounds} and $2K\log(1/\theta)=\beta/2$, we get
\begin{equation}\label{choice of M}
e^{-\beta e^{\gamma |a|}}\theta^{-2M(|a|)}\le e^{-\beta e^{\gamma |a|}}e^{\frac{\beta}{2} e^{\gamma |a|}}\theta^{-2} = \theta^{-2} e^{-\frac{\beta}{2} e^{\gamma |a|}}.
\end{equation}
With this choice of $M$, the second term in \eqref{final} decays double exponentially in $|a|$, while the first term decays only exponentially. Hence, we can find $C''>0$ such that, using the lower bound in \eqref{M bounds},
\begin{equation}\label{choice of C''}
 C\theta^{-2}e^{-\frac{\beta}{2} e^{\gamma |a|}}\le  C'' M(|a|)^{-\alpha} \le C'' \left( K e^{\gamma |a|}\right)^{-\alpha}
\end{equation}
for all $|a|\ge 0.$

Combining the estimates \eqref{final}--\eqref{choice of C''}, we obtain
\[
|\langle\mathcal{L}^a f, g\rangle_{L_2}| \le \hat{C}e^{-b |a|}\|f\|_{C^\alpha}\|g\|_{C^\alpha},
\]
where $\hat{C}=(C''+2C')K^{-\alpha}$ and $b=\alpha\gamma>0$.

It remains to treat $\alpha=1$. On $\mathbb{S}^1$ we have $C^{1}\subset C^{\alpha}$ for every $0<\alpha<1$, with $\|f\|_{C^{\alpha}}\le\|f\|_{C^{1}}$, so the case already proved applies to any $f,g\in C^1$ and gives
\[
|\langle\mathcal{L}^a f, g\rangle_{L_2}| \le \hat{C}e^{-\alpha\gamma |a|}\|f\|_{C^1}\|g\|_{C^1}
\]
for every $0<\alpha<1$.
\end{proof}

%% file: no_uniform_mixing_L_2.tex
\section{Mixing on \texorpdfstring{$L_2$}{L2}}\label{no-uniform-mixing}

\begin{lemma}[Isometries do not mix uniformly]\label{lm:isometry no decay}
Let $H\neq\{0\}$ be a Hilbert space and let $\mathcal{L}:H\to H$ be an isometry.
Then no sequence $\epsilon_n\to0$ can satisfy
\[
|\langle \mathcal{L}^{n}f,g\rangle|\le\epsilon_n\|f\|\,\|g\|
\qquad\text{for all } f,g\in H \text{ and } n\ge0 .
\]
\end{lemma}

\begin{proof}
We prove it by contradiction. Assume that there is a sequence $\epsilon_n\to0$
such that
\[
|\langle \mathcal{L}^{n}f,g\rangle|\le\epsilon_n\|f\|\,\|g\|
\]
for any $f,g\in H$ and $n\ge0$. Then for any $n\ge0$
\begin{equation}\label{L2 decay}
\|\mathcal{L}^{n}\|_{H\to H}
=\sup\bigl\{|\langle \mathcal{L}^{n}f,g\rangle|\ :\ f,g\in H,\ \|f\|=\|g\|=1\bigr\}
\le\epsilon_n\to0 .
\end{equation}
On the other hand, let us notice that if $\mathcal{L}$ is an isometry, then
$\mathcal{L}^{n}$ is an isometry for any $n\ge0$,
\[
\|\mathcal{L}^{n}f\|=\|\mathcal{L}(\mathcal{L}^{n-1}f)\|=\|\mathcal{L}^{n-1}f\|=\dots=\|f\| .
\]
So $\|\mathcal{L}^{n}\|_{H\to H}=1$ for any $n\ge0$, contradicting
\eqref{L2 decay}.
\end{proof}

\begin{corollary}[No uniform quantitative mixing on \texorpdfstring{$L_2$}{L2}]\label{no-uniform-mix}
Let $T$ be a transformation of a probability space $(X,\mu)$ preserving $\mu$,
and for $f,g\in L_2(X,\mu)$ put
\[
\mathcal{C}_{f,g}(n):=\left|\int_X\left(f\circ T^{n}\right)\overline{g}\,d\mu
-\int_X f\,d\mu\int_X\overline{g}\,d\mu\right| .
\]
Then no sequence $\epsilon_n\to0$ can satisfy
\[
\mathcal{C}_{f,g}(n)\le\epsilon_n\|f\|_{L_2}\|g\|_{L_2}
\qquad\text{for all } f,g\in L_2(X,\mu)\text{ and } n\ge0 .
\]
\end{corollary}

\begin{proof}
Let $\mathcal{L}f=f\circ T$ be the induced Koopman precomposition operator. Since
$T$ preserves $\mu$, we have $\|\mathcal{L}f\|_{L_2}=\|f\|_{L_2}$ and
$\int_X\mathcal{L}f\,d\mu=\int_Xf\,d\mu$, so $\mathcal{L}$ maps
$L^{0}_2(X,\mu)$ into itself and is an isometry of that space. For
$f,g\in L^{0}_2(X,\mu)$ both averages vanish and
$\mathcal{C}_{f,g}(n)=|\langle\mathcal{L}^{n}f,g\rangle_{L_2}|$. Now apply
Lemma~\ref{lm:isometry no decay} with $H=L^{0}_2(X,\mu)$.
\end{proof}

Note that this does not exclude mixing: no assumption on $T$ beyond
$\mu$-invariance is made, and for a mixing $T$ the correlations
$\mathcal{C}_{f,g}(n)$ of each fixed pair $f,g$ do tend to zero. What the
corollary excludes is a rate valid for all pairs at once. 
%, and suppose that\[L_2^0(X,\mu):=\left\{f\in L_2(X,\mu):\int_X f\,d\mu=0\right\}\neq\{0\}. it's not one point